\documentclass[11pt]{amsart}
\usepackage[margin=1in]{geometry}

\usepackage[utf8]{inputenc}

\usepackage{graphbox} 
\usepackage{mathtools, amsthm, amsfonts, amssymb, amsmath}
\usepackage{accents} 
\usepackage{microtype}
\usepackage{enumerate}

\usepackage[normalem]{ulem}

\usepackage{thmtools}
\usepackage[margin=1in]{geometry}
\usepackage{xcolor}
\usepackage{tikz}
\usepackage{tikz-cd}
\usetikzlibrary{patterns,external}
\usetikzlibrary{calc}
\usetikzlibrary{shapes.multipart}

\usepackage[numbers,sort]{natbib}

\usepackage[linesnumbered,commentsnumbered,ruled,vlined]{algorithm2e}

\usepackage[colorlinks]{hyperref}
\hypersetup{
  linkcolor=[rgb]{0.3,0.3,0.6},
  citecolor=[rgb]{0.2, 0.6, 0.2},
  urlcolor=[rgb]{0.6, 0.2, 0.2}
}
\makeindex

\renewcommand{\algocfautorefname}{Algorithm}

\definecolor{fondo}{rgb}{0.898,0.996,0.898}

\pgfkeys{/tikz/.cd,
  K/.store in=\K,
  K=1   
   }

\title{Quatroids and rational plane cubics}
\author[T. Brysiewicz]{Taylor Brysiewicz}
\address[T. Brysiewicz]{Department of Mathematics, University of Western Ontario, London, Canada (ORCID: 0000-0003-4272-5934)}
\email{tbrysiew@uwo.ca}
\author[F. Gesmundo]{Fulvio Gesmundo}
\address[F. Gesmundo]{Saarland Informatics Campus, Universität des Saarlandes, Saarbr\"ucken, Germany (ORCID: 0000-0001-6402-021X)}
\email{gesmundo@cs.uni-saarland.de}
\author[A. Steiner]{Avi Steiner}
\address[A. Steiner]{Fakultät für Mathematik, Technische Universität Chemnitz, Chemnitz, Germany (ORCID: 0000-0003-2095-9203)}
\email{avi.steiner@gmail.com}

\newcommand{\alert}[1]{{\color{red}#1}}
\newcommand{\mydef}[1]{{\color{blue}#1}}
\newcommand{\mycomment}[1]{}

\usepackage[all,cmtip]{xy}

\newcommand{\bfp}{\mathbf{p}}
\newcommand{\TC}{\textrm{TC}}
\newcommand{\T}{\textrm{T}}
\newcommand{\bbC}{\mathbb{C}}
\newcommand{\bbP}{\mathbb{P}}
\newcommand{\bbR}{\mathbb{R}}
\newcommand{\bbQ}{\mathbb{Q}}
\newcommand{\calP}{\mathcal{P}}
\newcommand{\calD}{\mathcal{D}}
\newcommand{\calC}{\mathcal{C}}

\newcommand{\calI}{\mathcal{I}}
\newcommand{\calJ}{\mathcal{J}}
\newcommand{\calM}{\mathcal{M}}
\newcommand{\calQ}{\mathcal{Q}}
\newcommand{\calS}{\mathcal{S}}
\newcommand{\calX}{\mathcal{X}}
\newcommand{\calZ}{\mathcal{Z}}

\newcommand{\mult}{\mathrm{mult}}
\newcommand{\imult}{\mathrm{imult}}

\newcommand{\frakS}{\mathfrak{S}}
\newcommand{\frakB}{\mathfrak{B}}
\newcommand{\frakQ}{\mathfrak{Q}}

\renewcommand{\bar}[1]{\overline{#1}}
\newcommand{\vvirg}{, \ldots ,}
\newcommand{\p}{\textbf{p}}

\theoremstyle{definition}
\newtheorem{theorem}{Theorem}[section]
\newtheorem{definition}[theorem]{Definition}
\newtheorem{lemma}[theorem]{Lemma} 
\newtheorem{corollary}[theorem]{Corollary}

\newtheorem{examplex}[theorem]{Example}
\newenvironment{example}
  {\pushQED{\qed}\examplex}
  {\popQED\endexamplex}
\newtheorem{remarkx}[theorem]{Remark}
\newenvironment{remark}
  {\pushQED{\qed}\remarkx}
  {\popQED\endremarkx}

\newtheorem{proposition}[theorem]{Proposition}

\DeclareMathOperator{\disc}{disc}

\newcommand{\Gr}{\mathrm{Gr}}

\newcommand{\PGL}{\mathrm{PGL}}

\usepackage[labelfont=bf, font={small}]{caption}[2005/07/16]

\makeatletter
\@namedef{subjclassname@2020}{%
  \textup{2020} Mathematics Subject Classification}
\makeatother


\begin{document}
\begin{abstract}
It is a classical result that there are $12$ (irreducible) rational cubic curves through $8$ generic points in $\mathbb{P}_{\mathbb{C}}^2$, but little is known about the non-generic cases. The space of $8$-point configurations is partitioned into strata depending on combinatorial objects we call quatroids, a higher-order version of representable matroids. We compute all $779777$ quatroids on eight distinct points in the plane, which produces a full description of the stratification. For each stratum, we generate several invariants, including the number of rational cubics through a generic configuration. As a byproduct of our investigation, we obtain a collection of results regarding the base loci of pencils of cubics and positive certificates for non-rationality.  
\end{abstract}

\maketitle

\section{Introduction}

\noindent In this article, we address the following planar interpolation problem:

\vspace{-0.2in}
\begin{center}
\begin{equation}\label{eq:Problem}
\textit{How many rational cubic curves pass through eight distinct points in } \mathbb{P}_{\mathbb{C}}^2 \textit{?} \tag{Problem 1}
\end{equation}
\end{center}
Famously, the answer to this enumerative problem is $12$ whenever the eight points are in \emph{generic} position (see \autoref{fig:TwelveCubics}). A modern proof of this classical result follows by an evaluation of Kontsevich's Formula \cite[Claim 5.2.1]{Kontsevich1994}.  We consider the \emph{non-generic} cases of\autoref{eq:Problem}.

\begin{figure}[!htpb]
\includegraphics[scale=0.8]{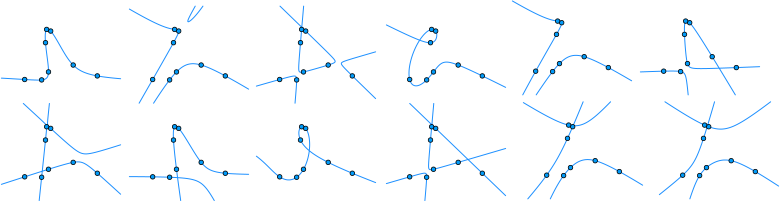}
\caption{Twelve rational cubics through eight generic points in the plane. Any appearances of reducibility or points coinciding with nodes are visual illusions.}
\label{fig:TwelveCubics}
\end{figure}

For almost all enumerative problems of interest, the count is  known only for generic parameters.
Determining the conditions of non-genericity for an enumerative problem is a hefty task; the challenge of describing the solution set over non-generic parameters is more difficult still.  
Such a full analysis has been achieved for few enumerative problems (e.g.\ \cite{Megyesi2003}). We advance the understanding of non-generic instances of\autoref{eq:Problem} through combinatorial objects we call \emph{quatroids}.

The parameter space of\autoref{eq:Problem} is the space $\mydef{\mathcal P} \subseteq (\mathbb{P}_{\mathbb{C}}^2)^8$ of configurations of eight \emph{distinct} points. We partition $\mathcal P$ into $779777$ \emph{quatroid strata}. Each stratum is a locally closed subset $\mydef{\mathcal S_{\mathcal Q}}\subset \mathcal P$ indexed by a pair $\mydef{\mathcal Q}=(\mydef{\mathcal I},\mydef{\mathcal J})$ of triples $\mathcal I$ and sextuples $\mathcal J$ of $\{1,2,\ldots,8\}$. Those pairs $\mathcal Q$ for which $\mathcal S_{\mathcal Q}$ is nonempty are called (representable) \mydef{quatroids}. A configuration $\mydef{\p}=(p_1,\ldots,p_8)$ is in $\mathcal S_{\mathcal Q}$ if the triples in $\calI$ index the triples of points in $\bfp$ lying on lines, and the sextuples in $\calJ$ do so for sextuples of points lying on (irreducible) conics. A similar construction was explored in \cite{Fie:PencilsCubicsEightPoints} in the context of M-curves and singular cubic curves interpolating convex configurations of points in $\bbP_{\mathbb{R}}^2$.

 The $779777$ quatroids appear in $125$ orbits under the natural action of the symmetric group $\mathfrak S_8$. We compute a representative quatroid ${\mathcal Q}$ of each orbit, along with several of its invariants (see \autoref{fig:Q10}). Included in this list is the number $\mydef{d_{\mathcal Q}}$ of rational cubics through a generic configuration in $\mathcal S_{\mathcal Q}$. 
\begin{figure}[!htpb]
\begin{tabular}{cccc}
\includegraphics[align=c,scale=0.6]{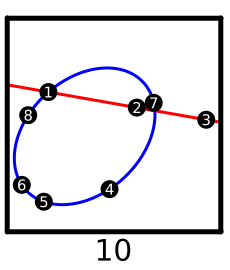} &&& \begin{tabular}{ll}Quatroid $\#$: 10 &
Orbit size: $168$ \\ 
Lines: {\color{red}{$\{123\}$}} &
Conics: {\color{blue}{$\{145678\}$}} \\
Reducibles: $\varnothing^1\mydef{\varnothing} $ &
$d_\mathcal Q$: $9$ \\
& \\
  $\mathbb{Q}$-Representative: & ${\tiny{\begin{bmatrix} 24 & 2 & -2 & 0 & 0 & 24 & -1 & 1 \\ 0 & 4 & 4 & 24 & 0 & 24 & -4 & 3 \\ 0 & -3 & -3 & 0 & 24 & 24 & 2 & -3 \end{bmatrix}}}$ \end{tabular}
\end{tabular}
\caption{An example of the data contained in the \nameref{appendix} and auxiliary files: A labeled illustration (left) of eight points realizing the tenth quatroid orbit $\mathcal Q_{10}=(\{123\},\{145678\})$. The size of the orbit of $\mathfrak S_8$ acting on $\mathcal Q_{10}$ is $168$. There is one reducible cubic through a generic configuration of $\mathcal Q_{10}$: the union of a conic and a secant line, indicated by the symbol $\varnothing^1$. The additional symbol $\mydef{\varnothing}$ indicates that this conic+secant appears with multiplicity one higher than expected. A rational representative is given.}
\label{fig:Q10}
\end{figure}

The main results about the stratification are summarized in the data contained in the \nameref{appendix}: 
\begin{itemize}
\item \autoref{fig:Quatroids1} and \autoref{fig:Quatroids2} illustrate each quatroid orbit via a real representative if available.
\item \autoref{tab:bigtable1} and \autoref{tab:bigtable2} enumerate each quatroid orbit. Included in these tables are their orbit sizes, combinatorial descriptions, descriptions of the  reducible cubics through generic representatives, and the counts of rational cubics through a generic representatives.
\item \autoref{tab:software} lists and describes the auxiliary files used to obtain the results in the paper.
\end{itemize} 
The code and auxiliary files may be found at
\begin{center}
\href{https://mathrepo.mis.mpg.de/QuatroidsAndRationalPlaneCubics}{https://mathrepo.mis.mpg.de/QuatroidsAndRationalPlaneCubics}
\end{center}

\autoref{fig:Quatroids1}/
\autoref{fig:Quatroids2} and \autoref{tab:bigtable1}/\autoref{tab:bigtable2} contain $126$ entries, rather than $125$. These correspond to the  $126$ orbits of \emph{candidate quatroids}  computed in \autoref{sec:Quatroids}. The $63^{\text{rd}}$ entry in each describes a collection of linear and quadratic dependencies which cannot be realized by any configuration of eight distinct points in $\mathbb{P}_{\mathbb{C}}^2$.

An enumerative problem such as\autoref{eq:Problem} can be viewed as a \emph{branched cover}, in the sense of \cite[Section 2.2]{NumericalNonlinearAlgebra}. Specifically, let $\mydef{\mathbb{P}S^3\mathbb{C}^3}$ be the projectivization of the space of homogeneous ternary cubics and consider the following two subsets: 
\begin{align*}
\mydef{\mathcal C} &= \{C \in \mathbb{P}S^3\mathbb{C}^3 \mid C \text{ is rational and irreducible}\}, \\
\mydef{\mathcal D} &= \{C \in \mathbb{P}S^3\mathbb{C}^3 \mid C \text{ is singular}\}.
\end{align*}
It is a classical fact that $\mathcal C$ is an open subvariety of the degree $12$ irreducible hypersurface $\mathcal D\subseteq \mathbb{P}S^3\mathbb{C}^3$  called the \mydef{discriminant} (see, e.g., \cite[Example I.4.15]{GKZ:DiscResMultDet}). 

\autoref{eq:Problem} pertains to rational cubics, but it is useful to consider the analogous problem for singular cubics. Hence, we consider incidence correspondences
\begin{align*}
\mydef{\mathcal X} &= \{(\p,C) \in \mathcal P \times \mathcal C \mid C(p_i) = 0, i=1,\ldots,8\}, \\
\mydef{\overline{ \mathcal X}} &= \{(\p,C) \in \mathcal P \times \mathcal D \mid C(p_i) = 0, i=1,\ldots,8\}.
\end{align*}
The natural projection maps $\pi: \calX \to \calP$ and $\bar{\pi}: \bar{\calX} \to \calP$ define two branched covers over $\calP$; clearly $\bar{\calX}$ is the closure of $\calX$ in $\calP \times \bbP S^3 \bbC^3$ and $\pi$ is the composition of the open inclusion map $\mathcal C \hookrightarrow \mathcal D$ with $\bar{\pi}$. With this setup, solving an instance of either enumerative problem over a configuration $\p \in \mathcal P$ is the same as computing the fibers $\pi^{-1}(\p)$ or $\overline{\pi}^{-1}(\p)$ of the respective branched cover.

For fixed $\p \in \mathcal P$, the equations $C(p_i)=0$ define a linear space $\mydef{L_{\p}} \subset \mathbb{P}S^3\mathbb{C}^3$ on the space of ternary cubics; explicitly, $L_\bfp$ is the projectivization of the homogeneous component of degree $3$ of the ideal of the points $\bfp$. If the conditions $C(p_i) = 0$ are linearly independent, then the dimension of $L_{\p}$ is $ \dim (\bbP S^3 \bbC^{10}) - 8 = 1$: in this case $L_\bfp$ is a pencil. Hence, singular cubics through $\p \in \mathcal P$ correspond to points in $L_{\p} \cap \mathcal D$ and the analysis of $\overline{\pi}$ is equivalent to the analysis of how pencils of cubics in $\mathbb{P}S^3\mathbb{C}^3$ meet the discriminant hypersurface $\mathcal D$.

In \autoref{sec:AGandCubics} we outline the relevant algebro-geometric geometric facts underlying our analysis of the branched covers $\pi$ and $\overline{\pi}$. We focus on the geometry of cubics. We characterize the point configurations $\p \in \mathcal P$ for which there are infinitely many singular cubics through $\p$ (\autoref{prop:FibreIsFinite}) as well as those configurations for which there are fewer than $12$ rational cubics through $\p$ (\autoref{thm:Characterization}). 

\autoref{thm:Characterization} forms the main idea behind the stratification of $\mathcal P$. It states that there are fewer than $12$ rational cubics through $\p \in \mathcal P$ if and only if $\p$ represents a non-uniform quatroid.  In  \autoref{sec:Quatroids}, we define (representable) quatroids (\autoref{def:Quatroid}) of eight distinct points in $\mathbb{P}_\mathbb{C}^2$ and discuss their relationship to matroids. We compute the set of \emph{candidate quatroids} $\mydef{\mathfrak Q}$, a superset of all quatroids, using \autoref{algo:allConicalExtensions}. We achieve this by making use of two necessary criteria for $\mathcal Q$ to be a quatroid: the underlying matroid is representable and $\mathcal Q$ satisfies \emph{B\'ezout's weak criteria} (\autoref{def:WeakCriteria}). For each candidate quatroid $\mathcal Q$, we determine if $\mathcal Q$ is representable over $\mathbb{Q}$, $\mathbb{R}$, or $\mathbb{C}$. We give explicit representations in the auxiliary file \texttt{RationalRepresentatives.txt}. We show that candidate $\mathcal Q_{63}$ is not representable over $\mathbb{C}$, and as a consequence, we obtain the true list of $125$ orbits of quatroids.

In \autoref{sec:IrreducibilityOfStrata} we show that all quatroids have irreducible realization spaces except for $\mathcal Q_{41}$, which has two irreducible components. This result justifies our use of the word \emph{generic} when considering points on these strata. 

Since the number $\mydef{d_{\p}}$ of rational cubics through a configuration $\p$ is upper semicontinuous, counting the rational cubics interpolating any representative configuration in $\mathcal S_{\mathcal Q}$ provides a lower bound on the number $d_{\mathcal Q}$ of rational cubics through a generic point of $\mathcal S_{\mathcal Q}$. This is done in  \autoref{sec:lowerbounds}; specifically, we compute these bounds in  \autoref{thm:lowerbounds}  using our rational representatives.

In \autoref{sec:expectedintersections} we turn toward showing that our lower bounds  on each $d_{\mathcal Q}$ are tight. We compute upper bounds for each $d_{\mathcal Q}$ by bounding the number of reducible cubics through any configuration $\p \in \mathcal S_{\mathcal Q}$, using the multiplicity of each reducible cubic as a point on $\mathcal D$.  Our lower bounds agree with our upper bounds in all but $24$ cases (see  \autoref{thm:upperboundsnottight}). In these remaining $24$ cases, the bound is off by one. We account for this discrepency by showing that configurations which represent any of these $24$ quatroid orbits must correspond to lines $L_{\p} \subset \mathbb{P}S^3\mathbb{C}^3$ which are tangent to a branch of $\mathcal D$ through a reducible cubic. Doing so, we achieve our main result \autoref{thm:MainTheorem} of determining $d_{\mathcal Q}$.  

 In \autoref{sec:conclusion} we collect several immediate consequences of our computations and outline challenges to refining our stratification. We explain how quatroids offer positive certificates for the non-rationality of cubics. In this section, we also compute the number of rational quartic curves through two special configurations using numerical algebraic geometry. Finally, we discuss examples which show that a complete stratification may be within reach, which would fully answer\autoref{eq:Problem}.

\begin{remark}
The computations supporting  several of our main theorems were first performed numerically using \texttt{HomotopyContinuation.jl} \cite{BreTim:HomotopyContinuation} in \texttt{julia} \cite{julia}. This includes 
\begin{itemize}
\item The computation of generic points on each (nonempty) quatroid stratum. 
\item The computation of \textbf{real} points on each stratum other than $\mathcal S_{\mathcal Q_{41}}$ and $\mathcal S_{\mathcal Q_{63}}$.
\item The computation of $d_{\mathcal Q}$ for each candidate quatroid $\mathcal Q$
\end{itemize}
For the sake of accessibility, brevity, and clarity, we decided to list explicit $\mathbb{Q}$-representatives instead of relying on subtle claims which depend upon numerical certification methods. Ultimately, despite its important role in developing our intuition about the present problem, no proof in this manuscript comes from a numerical computation. We showcase the power of numerical methods in \autoref{sec:conclusion} when discussing the possibility of extending our results to quartics.
\end{remark}

\section*{Acknowledgements}
We are grateful for the helpful conversations with Lukas K\"uhne during the early stages of this project.
The first author is partially supported by an NSERC Discovery Grant (Canada). The third author is partially supported by DFG Emmy-Noether-Fellowship RE 3567/1-1.

\section{Background on Algebraic Geometry and Cubics}\label{sec:AGandCubics}
In this section, we collect some standard facts about the geometry of points in the plane and cubic curves through them. Let $\p \in \mathcal P $ be a configuration of eight \emph{distinct} points in $ \mathbb{P}_\mathbb{C}^2$ and $L_{\p} \subseteq \bbP S^3 \bbC^3$ be the linear space of cubics vanishing at $\p$. Write \mydef{$Z(L_{\p})$} for the scheme in $\mathbb{P}_{\mathbb{C}}^2$ defined by the ideal generated by $L_{\p}$, called the \mydef{base locus} of $L_{\p}$. A $1$-dimensional subspace of $\mathbb{P}S^3\mathbb{C}^3$ is called a \mydef{pencil} of cubics. Plane curves of degree one, two, and three are called \mydef{lines}, \mydef{quadrics}, and \mydef{cubics}, respectively. An irreducible quadric is a \mydef{conic} and a \mydef{rational} cubic is a singular {irreducible} cubic.

A classical consequence of B\'ezout's Theorem \cite[Section 5]{Fulton2008} guarantees that in the cases of interest for\autoref{eq:Problem} the linear space $L_\bfp$ is $1$-dimensional. We include a proof for completeness.

\begin{proposition}
\label{prop:FibreIsFinite} Let $\p \in \mathcal P$. The following are equivalent.
\begin{enumerate}
\item The set $\overline{\pi}^{-1}(\p) = L_{\p} \cap \mathcal D$ is finite,
\item There is an irreducible cubic through $\p$,
\item No four points of $\p$ are on a line and no seven points of $\p$ are on a conic,
\item The residue of the scheme $Z(L_\bfp)$ with respect to $\bfp$ is a reduced point $p_9$, called the \mydef{Cayley--Bacharach point} of $\p$. 
\end{enumerate}
\end{proposition}
\begin{proof}
Observe preliminarily that by B\'ezout's Theorem, the intersection of two cubics is either of positive dimension or it is a $0$-dimensional scheme of degree $9$ in $\bbP_{\mathbb{C}}^2$. We show part (2) is equivalent to all others.

\noindent\underline{Part (1)$\implies$ Part (2)}: The linear space $L_{\p}$ has dimension at least $1$ because it is defined by eight linear conditions on $\bbP S^3 \bbC^3 \simeq \bbP_{\mathbb{C}}^9$. Hence, if $L_{\p} \cap \mathcal D$ is finite, there exists some $C \in L_{\p} \backslash \mathcal D$, that is, an irreducible cubic through $\p$. 

\noindent\underline{$\neg$Part (1)$\implies \neg$Part (2)}: If the intersection $L_{\p} \cap \mathcal D$ is infinite, then either $L_{\bfp} \subseteq \calD$ or $\dim L_{\bfp} \geq 2$. 

In the first case, $L_{\bfp}$ is a pencil of singular cubics. By Bertini's Theorem \cite[p.137]{GrifHar:PrinciplesAlgebraicGeometry} the generic element of $L_\bfp$ is smooth away from the base locus of $L_\bfp$, so there is a point $p\in \bbP_{\mathbb{C}}^2$ such that all elements of $L_\bfp$ are singular at $p$. Hence, the fat point $2p$ supported at $p$ is contained in $Z(L_{\p})$. Since $\deg(2p) = 3$, and the base locus of $L_\bfp$ contains at least $7$ additional distinct points, this implies $Z(L_{\p})$ is not a $0$-dimensional scheme of degree $9$ and therefore must contain a positive dimensional component. This component is a curve in $\bbP_{\mathbb{C}}^2$ and all elements of $L_\bfp$ are multiples of the polynomial $f$ defining it. The polynomial $f$ has degree smaller than three since otherwise $L_{\p}$ is just a point. Hence, $L_\bfp$ contains no irreducible element.

Now suppose $\dim L_{\bfp} \geq 2$. If the points of $\bfp$ lie on a line or on a conic, then $L_{\bfp}$ is generated by the corresponding equation of degree $1$ or $2$ so it does not contain irreducible elements. Indeed, if $L_{\bfp}$ contains an equation not generated by the ones of lower degree, then $\bfp$ would be contained in a complete intersection of type $(1,3)$ or $(2,3)$, which contains at most $6$ points. In particular, we may assume the ideal of $\bfp$ has no elements of degree $1$ or $2$. We apply \cite[Thm. 3.6]{BigGerMig:GeometricConsequencesMacaulay}, and we refer to \cite[Thm. 2.3]{ChiGes:DecompTerraciniCubic} for a simpler statement: if $\dim L_{\bfp} \geq 2$, and the ideal of $\bfp$ contains no elements of degree $2$, then the h-vector of $\bfp$ is $(1,2,3,1,1)$; in this case $\bfp$ has five points on a line $\ell=0$. Let $q_1,q_2,q_3$ be three linearly independent quadrics vanishing on the (at most) three points of $\bfp$ not lying on $\ell$, then $L_\bfp = \langle \ell q_1, \ell q_2, \ell q_3\rangle$. As before, this shows that all elements of $L_{\bfp}$ share a common factor. 

\noindent\underline{Part (2)$\implies$Part (3)}: By B\'ezout's Theorem, an irreducible cubic contains no four points on a line or seven on a conic.

\noindent\underline{Part (3)$\implies$Part (2)}: We showed above that if $\dim L_{\bfp} \geq 2$ then either $\bfp$ lies on a line or a conic, or at least five elements of $\bfp$ lie on a line. Therefore, assume $\dim L_{\bfp} = 1$ and suppose all of its elements are reducible. We will show that $\bfp$ has at least four points on a line or at least seven on a conic. As before, by Bertini's Theorem, all elements of $L_{\bfp}$ have at least one common singular point, and we deduce that either a line or a conic is contained in the base locus of $L_{\bfp}$.

Suppose the line $\{ \ell = 0\}$ is a line in the base locus so that $L_{\bfp} = \langle \ell q_1,\ell q_2 \rangle$. We consider two cases. If $\{ q_1 = q_2 = 0\}$ is $0$-dimensional, then it contains at most four points of $\bfp$, and at least four points of $\bfp$ lie on $\{ \ell = 0\}$. If $\{ q_1 = q_2 = 0\}$ is positive dimensional then $q_1 = \ell' \ell_1$ and $q_2 = \ell' \ell_2$: then $\bfp \subseteq \{\ell=0\} \cup \{\ell' = 0\} \cup \{\ell_1=\ell_2 = 0\}$, so that at least four points lie on a line. 

Suppose the conic $\{q = 0\}$ is a conic in the base locus, so that $L_{\bfp} = \langle q \ell_1 ,q \ell_2 \rangle$. Then $\bfp \subseteq \{ q =0\} \cup \{ \ell_1 = \ell_2 = 0\}$, showing that at least seven points lie on a conic.

\noindent\underline{Part (2)$\implies$Part (4)}: This is the classical Cayley--Bacharach Theorem. If $L_{\bfp}$ contains at least one irreducible element, then the base locus of $L_\bfp$ is a $0$-dimensional scheme of degree $9$. The residue is, by definition, cut out by the ideal $(L_{\bfp}) : I(\bfp)$, where $(L_\bfp)$ denotes the ideal generated by the cubics of $L_\bfp$. Since $\deg( Z(L_\bfp)) = 9$, and $\bfp$ consists of $8$ points with $\bfp \subseteq Z(L_\bfp)$, we have that the degree of the residue is $1$. This shows that it is a reduced point $p_9$.

\noindent\underline{Part (4)$\implies$Part (2)}: If the residue is a reduced point, then $Z(L_\bfp)$ is $0$-dimensional. In this case, the same arguments used above show that the generic element of $L_{\bfp}$ is an irreducible cubic.
\end{proof}

A consequence of \autoref{prop:FibreIsFinite} is that if either $L_\bfp$ is not a pencil, or if $L_\bfp \subseteq \calD$, then there are no rational cubics through $\bfp$. When $L_{\p}$ is a pencil intersecting $\mathcal D$ in finitely many points, write $\mydef{\textrm{imult}_C(L_{\p}, \mathcal D)}$ for the \mydef{intersection multiplicity} of $L_\bfp$ and $\calD$ at a point $C \in L_{\p} \cap \mathcal D$ (see \cite[Section 3]{Fulton2008}): in other words, $\mydef{\textrm{imult}_C(L_{\p}, \mathcal D)}$ is the degree of the component supported at $C$ in the $0$-dimensional scheme $L_\bfp \cap \calD$. The related concept of the \mydef{multiplicity} of a point $C \in \mathcal D$ is defined as follows:
\begin{equation}\label{eq:multiplicity}
\mydef{\textrm{mult}_{\mathcal D}(C)}=\min \{ \mathrm{imult}_C(L, \mathcal D) \mid L \in \mathrm{Gr}(2,S^3 \bbC^3) \text{ with } C \in L\},
\end{equation} 
where $\mydef{\Gr(2,S^3 \bbC^3)}$ is the Grassmannian of projective lines in $\mathbb{P}S^3\mathbb{C}^3$. By a semicontinuity argument, a generic line $L$ through $C$ realizes $\imult_C(L,\calD) = \mult_\calD(C)$.

Our main interest for this work is the following quantity. Recall that $\calC$ is the subset of $\calD$ consisting of irreducible cubics.
\begin{definition}\label{def:dp}
For $\p \in \mathcal P$, define 
\[
\mydef{d_{\p}} = \sum_{C \in L_{\p} \cap \mathcal C} \textrm{imult}_C(L_\p,\mathcal D)
\]
to be the number of rational cubics through $\p$ \emph{counted with multiplicity}.
\end{definition}

The following is a direct consequence of \autoref{prop:FibreIsFinite}.

\begin{corollary}
\label{cor:alwaysfinitelymanyrationals}
For $\p \in \mathcal P$ the value $d_{\p}$ is finite.
\end{corollary}

The set-difference of $L_{\p} \cap \mathcal D$ and $L_{\p} \cap \mathcal C$ is comprised of the \emph{reducible} cubics through $\p$. Taking intersection multiplicity into account allows for a stronger correspondence. Hence, for any $\p \in \mathcal P$ satisfying  the conditions of \autoref{prop:FibreIsFinite}, we define 
$$\mydef{r_{\p}} = \sum_{\text{reducible }C \in L_{\p} \cap \mathcal D} \textrm{imult}_C(L_{\p},\mathcal D)=\sum_{C \in L_{\p} \cap \mathcal D \backslash \mathcal C} \textrm{imult}_C(L_{\p},\mathcal D).$$
If $\bfp$ fails to satisfy the conditions in \autoref{prop:FibreIsFinite}, we set $r_\bfp = \infty$. The following proposition justifies this convention and expresses how $r_{\p}$ measures the difference of  $L_{\p} \cap \mathcal D$ and $L_{\p} \cap \mathcal C$. 

\begin{proposition}
\label{prop:FibreHas12Points}
Let $\p \in \mathcal P$ satisfy any of the equivalent conditions in  \autoref{prop:FibreIsFinite}. Then 
\[
12=\sum_{C \in L_{\p} \cap \mathcal D} \imult_C(L_{\p}, \mathcal D)  =r_{\p}+\sum_{C \in L_{\p} \cap \mathcal C} \imult_C(L_{\p},\mathcal D) = r_{\p} + d_{\p}. 
\]
If $\bfp$ does not satisfy the condition of \autoref{prop:FibreIsFinite}, then there are infinitely many reducible cubics through $\p$ and moreover $d_{\p}=0$.
\end{proposition}
\begin{proof}If $\p$ satisfies  \autoref{prop:FibreIsFinite}, then $L_{\p}$ is a line and $L_{\p} \cap \mathcal D$ is finite. 
Since $\mathcal D$ is a projective hypersurface of degree $12$, the intersection $L_\bfp \cap \mathcal D$ is a $0$-dimensional scheme of degree $12$. This shows the left equality. The middle equality follows from the definition of $r_{\p}$. The last equality is the definition of $d_{\p}$. The final statement  follows from \autoref{prop:FibreIsFinite}.
\end{proof}
 \autoref{prop:FibreHas12Points} answers the problem of counting (with multiplicity) the \emph{singular} cubics through eight distinct points: when finite, this number is $12$. The main insight of \autoref{prop:FibreHas12Points} is that the number $d_{\p}$ is determined as the difference $12 - r_\bfp$, and $r_{\p}$ is the number of reducible cubics through $\p$ counted with multiplicity. This main idea anchors our story.
\begin{theorem}
\label{thm:Characterization}
Let $\p \in \mathcal P$. The following are equivalent:
\begin{enumerate}
\item Either $\p$ has three points on a line or six points on a conic,
\item There is a reducible cubic through $\p$, i.e. $r_{\p}>0$,
\item The inclusion $\pi^{-1}(\p) \subsetneq \overline{\pi}^{-1}(\p) $ is strict,
\item Counted with multiplicity, there are fewer than $12$ rational cubics through $\p$, i.e. $d_{\p}<12$.
\end{enumerate}
\end{theorem}
\begin{proof}
If $L_\bfp \cap \calD$ is infinite, the result follows directly from \autoref{prop:FibreIsFinite}. Hence, suppose that $L_\bfp$ is a pencil and $L_\bfp \cap \calD$ is finite.

\noindent\underline{Part (1)$\iff$Part (2)}: If $\p$ has three points on a line $\ell$ then the union of $\ell$ and a conic through the remaining five points is a reducible cubic through $\p$. Similarly, if six points are on a conic, then the union of that conic and the line containing the other two is a reducible cubic through $\p$. Conversely, if there is a reducible cubic through $\p$, the pigeonhole principle implies that either three are on a line or six are on a conic since reducible cubics are either unions of three lines or a line-conic union. 

\noindent\underline{Part (2)$\iff$Part (3)}: This is immediate from the definitions.
\noindent\underline{Part (4)$\implies$Part (3)}:
Since the fibre $\overline{\pi}^{-1}(\p)$  is finite, by  \autoref{prop:FibreHas12Points} it  consists of $12$ cubics counting multiplicity. From the hypothesis, $\pi^{-1}(\p)$ does not, so part (3) is true.

\noindent\underline{Part (2) $\implies$ Part (4)}:
Since the fibre is finite, $\overline{\pi}^{-1}(\p)$ consists of $12$ points counting multiplicity by  \autoref{prop:FibreHas12Points}. Since part (2) is true, at least one cubic in $\overline{\pi}^{-1}(\bfp)$ is reducible, and so there are fewer than $12$ rational cubics through $\p$ counting multiplicity.
\end{proof}

\section{Quatroids}
\label{sec:Quatroids}
As a consequence of the results in the previous section, any configuration $\p \in \mathcal P$ admitting fewer than $12$ rational cubics must lie on at least one of the following two types of hypersurfaces in $\mathcal P$:
\[
\mydef{X_{i_1i_2i_3}} = \{\p \in \mathcal P \mid p_{i_1}\ldots p_{i_3} \text{ are on a line}\}, \quad \quad \mydef{Y_{i_1i_2\cdots i_6}} = \{\p \in \mathcal P \mid p_{i_1}\ldots p_{i_6} \text{ are on a conic}\}. 
\]
Elements of the intersection lattice of these ${{8}\choose{3}} + {{8}\choose{6}} 
=84$ hypersurfaces are indexed by pairs $\mathcal Q=(\mathcal I, \mathcal J)$ of triples $\mathcal I$ and sextuples $\mathcal J$ of $\{1,2,\ldots,8\}$. Define
\[
\mydef{\mathcal Z_{\mathcal Q}} = \biggl(\bigcap_{I \in \mathcal I} X_I \biggr) \cap \biggl(\bigcap_{J \in \mathcal J} Y_J \biggr)  \quad \text{ and } \quad \mydef{\mathcal S_{\mathcal Q}} = \calZ_\calQ \setminus \bigcup_{\calQ'} \calZ_{\calQ'}
\]
where the union ranges over all $\calQ'$ such that $\calZ_{\calQ'} \subsetneq \calZ_{\calQ}$.

By construction, the stratification $\mydef{\mathcal S}$ of nonempty loci $\mathcal S_{\mathcal Q}$ partitions the parameter space $\mathcal P$. The strata, and thus their indices, are naturally ordered by $\mathcal Q \mydef{\leq} \mathcal Q'$ if and only if $\mathcal S_{\mathcal Q} \subseteq \overline{\mathcal S_{\mathcal Q'}}$.

 With this setup, we introduce the main combinatorial object of interest. Although we focus on fields of characteristic zero, we give the definition for any field $\mydef{\mathbb{K}}$. Write $
 \mydef{\mathcal P_{\mathbb{K}}} $ for the open subset $8$-tuples of distinct points in $(\mathbb{P}_{\mathbb{K}}^2)^8$.
\begin{definition} 
\label{def:Quatroid} Let $\mathbb{K}$ be a field and let $\mathcal I$ and $\mathcal J$ be collections of triples and sextuples of $\{1,2,\ldots,8\}$, respectively. The pair $\mathcal Q=(\mathcal I,\mathcal J)$ is a \mydef{($\mathbb{K}$-representable) quatroid} if there exists a configuration $\p \in \mathcal P_{\mathbb{K}}$ such that $\bfp \in \mathcal S_{\mathcal Q}$. In this case, we say $\mathcal Q$ is \mydef{representable over $\mathbb{K}$}. Equivalently, there exists $\p \in \mathcal P_{\mathbb{K}}$ such that
\begin{itemize}
\item Every triple of points in $\p$ indexed by an element of $\mathcal I$ lies on a line,
\item Every sextuple of points in $\p$ indexed by an element of $\mathcal J$ lies on a conic,
\item No other triple of points in $\p$ lies on a line,
\item No other sextuple of points in $\p$ lies on a conic.
\end{itemize}
Such a configuration $\p$ is said to \mydef{represent} $\mathcal Q$ over $\mathbb{K}$. The set $\calS_\calQ$ is called the \mydef{realization space} of $\mathcal Q$ over $\mathbb{K}$.
\end{definition}

For any $\p  \in \mathcal P$ and $S \subseteq \{1,2,\ldots,8\}$ we write $\mydef{\ell_S}$ for the line through $\{p_s\}_{s \in S}$ and $\mydef{q_S}$ for the conic through $\{p_s\}_{s \in S}$ when such a curve exists and is unique.  Given a configuration $\p$ satisfying any condition in \autoref{prop:FibreIsFinite}, the quatroid represented by $\p$ exactly tracks the reducible cubics through $\p$, as detailed in the following result.

\begin{lemma}
\label{lem:reducibleCubicsThroughp}
Let $\p$ be a configuration satisfying \autoref{prop:FibreIsFinite} and representing a quatroid $\mathcal Q=(\mathcal I, \mathcal J)$. A reducible cubic through $\p$ has one of the forms:
\begin{enumerate}
\item[(a)] $\ell_I\cdot q_{I^c}$ for $I \in \mathcal I$,
\item[(b)] $\ell_{I_1}\cdot \ell_{I_2} \cdot \ell_{(I_1 \cup I_2)^c}$ for \emph{disjoint} $I_1,I_2 \in \mathcal I$,
\item[(c)] $q_{J}\cdot \ell_{J^c}$ for $J \in \mathcal J$.
\end{enumerate}
\end{lemma}
\begin{proof}
If there is a reducible cubic through $\p$, by the pigeonhole principle, and \autoref{prop:FibreIsFinite}, either exactly six points of $\bfp$ lie on a conic or exactly three points of $\bfp$ lie on a line. 

In the first case, the reducible cubic must be the union of the conic containing six points and the line containing the remaining two. The six points on the conic must be indexed by some $J \in \mathcal J$, so the reducible cubic has the form (c). 

In the second case, the reducible cubic must be the union of the line containing three points and the \emph{quadric} containing the remaining five. Following a similar argument as in the first case, if the quadric is a irreducible, then the cubic has the form (a) and otherwise it has the form (b). 
\end{proof}

\begin{remark}
\label{rem:OrderIsCombinatorial}
An consequence of \autoref{lem:reducibleCubicsThroughp} is that the ordering $(\calI,\calJ)\leq (\calI',\calJ')$ may be expressed combinatorially. Given our conventions, the full characterization is rather technical. However, we list two simple sufficient conditions: 
\begin{enumerate}
\item[(a)] $(\mathcal I,\mathcal J) \leq (\mathcal I', \mathcal J')$ if $\mathcal I' \subseteq \mathcal I$ and $\mathcal J' \subseteq \mathcal J$,
 \item[(b)] $(\mathcal I, \mathcal J) \leq (\mathcal I', \mathcal J')$ if $\mathcal I' \subseteq \mathcal I$ and $\mathcal J' \subseteq \mathcal J \cup \mathcal I^2$, where $\calI^2$ denotes the set of subsets obtained as union of two elements of $\calI$. 
\end{enumerate}
Condition (b) completely characterizes $\leq$ restricted to a class of quatroids called \emph{B\'ezoutian quatroids} defined in \autoref{def:Bezoutian}.
\end{remark}

We now digress into several remarks on the choice of the term \emph{quatroid}.
\begin{remark}[Inspiration for the term]
The term \emph{(representable) quatroid} is inspired by the term \emph{representable matroid}. Quatroids partially extend the concept of matroids; representable matroids encode the affine linear dependencies among points, whereas representable quatroids track linear and quadratic dependencies. There are several standard references for matroids. We suggest \cite{Oxley}.
\end{remark}

\begin{remark}[Abstraction of quatroids]
There are several natural ways to extend  \autoref{def:Quatroid} to include point configurations with more points or in higher dimensional spaces. We use the specific definition in \autoref{def:Quatroid} for the sake of brevity since this is all we need for the purpose of the paper. We leave the challenge of defining a \emph{non-representable quatroid} for future work.  If $\mathcal Q=(\mathcal I,\mathcal J)$ is not a quatroid, we merely call it a \mydef{pair}.
\end{remark}

\begin{remark}[Pascal's Theorem to linearize conic conditions]
Pascal's Theorem states that six points are on a conic if and only if the three auxiliary intersection points of a hexagram as shown  in \autoref{fig:hexagrammummysticum} lie on a line, called the \mydef{Pascal line}. One may attempt to convert conic dependencies into linear dependencies using Pascal's Theorem. This result may be applied to $60$ orderings of the six points, introducing $15$ hexagram lines, $45$ auxiliary points, and $60$ Pascal lines. The right-hand image in \autoref{fig:hexagrammummysticum} illustrates a fraction of how complicated the entire arrangement gets, and thus, the benefit of recording the conic conditions directly.
\end{remark}

\begin{figure}[!htpb]
\includegraphics[scale=0.3]{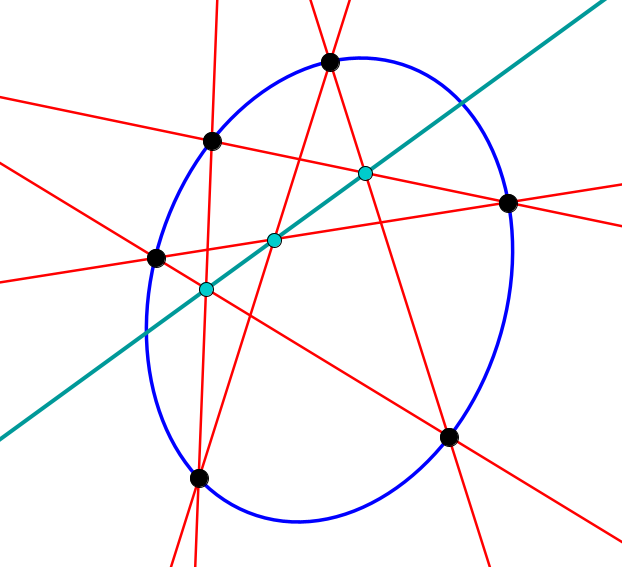} \quad \quad \quad \quad 
\includegraphics[scale=0.23]{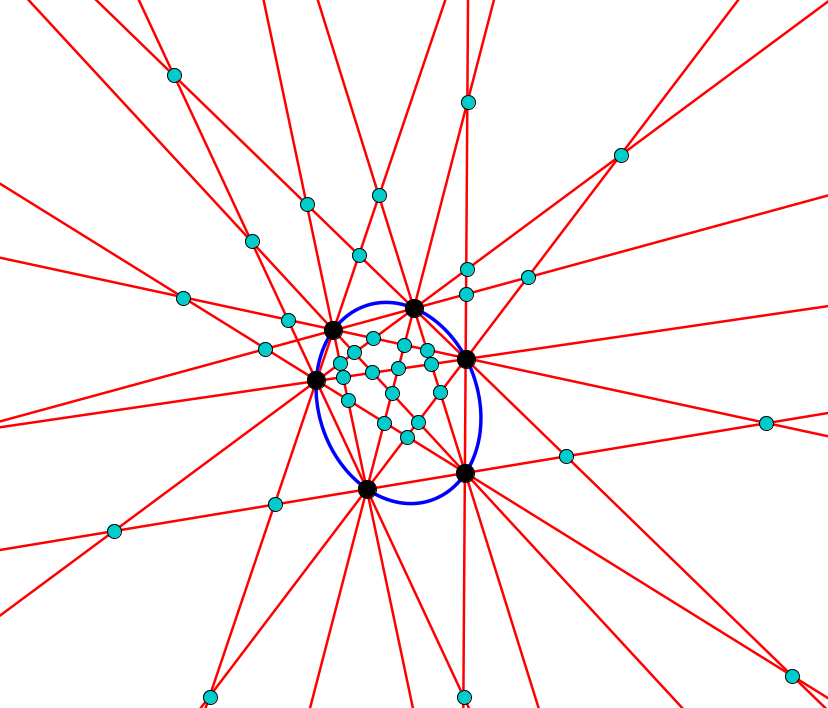}
\caption{(Left) Illustration of Pascal's Theorem. (Right) The $15$ hexagram lines and a subset of the $45$ auxiliary points (in view) obtained by applying Pascal's Theorem $60$ times. The Pascal lines are not drawn.}
\label{fig:hexagrammummysticum}
\end{figure}

\begin{remark}[Simplicity]
Our definition of a quatroid requires that a representative point configuration $\p$ consists of \emph{distinct} points. Such a restriction for matroids means that the matroid is \emph{simple}: the matroid has no parallel elements (repeated points) and the matroid has no loops (consists of points in $\mathbb{P}_{\mathbb{C}}^2$). We suspect that an effective definition of \emph{quatroid} which covers the case of repeated points would need to involve additional information about this repetition. Possibly, there is a way to combinatorially track some information encoded in the \emph{Hilbert scheme} of eight points in $\mathbb{P}_{\mathbb{C}}^2$. 
\end{remark}

\begin{remark}
For the remainder of this article, we assume that $\mathbb{K}$ is $\mathbb{Q}$, $\mathbb{R}$, or $\mathbb{C}$.
\end{remark}

Given \autoref{def:Quatroid}, a natural question emerges:
\begin{center}
\emph{Which pairs $\mathcal Q=(\mathcal I, \mathcal J)$ are representable  quatroids?}
\end{center}
\noindent We address this problem by providing two necessary conditions on $\mathcal Q$ to be a quatroid, finding all $\mathcal Q$ satisfying those conditions, and subsequently identifying which of those candidates are quatroids by either exhibiting an element $\bfp \in \calP$ representing it or proving that no such $\bfp$ exists.

\subsection{Necessary conditions for being a quatroid} 
The \mydef{underlying matroid} of a quatroid $\calQ = (\calI , \calJ)$ is the matroid $\mydef{\mathcal M_{\mathcal I}}$ whose nonbases of size three are given by $\mathcal I$. The realization space of $\calM_\calI$ is $\mydef{\mathcal S_{\mathcal I}}=\bigcup_{\mathcal J} \mathcal S_{(\mathcal I,\mathcal J)}$ because if $\bfp$ represents $\calQ$, then it represents $\calM_{\calI}$ as well. 
The matroid $\mathcal M_{\mathcal I}$ is \mydef{representable} if $\mathcal S_{\mathcal I}$ is nonempty.  We call a pair $\mathcal Q=(\mathcal I,\mathcal J)$  \mydef{matroidal} whenever $\mathcal M_{\mathcal I}$ is representable. A representable quatroid $\mathcal Q$ is necessarily matroidal. Additional necessary conditions for $\mathcal Q$ to be representable rely upon B\'ezout's Theorem.

\begin{definition}
\label{def:WeakCriteria}
A pair $\mathcal Q=(\mathcal I,\mathcal J)$ satisfies \mydef{B\'ezout's weak criteria} if 
\begin{enumerate}
\item ``A line and conic meet in at most two points'':
\[
 |I \cap J|\leq 2 \text{ for all } (I,J) \in \mathcal I \times \mathcal J.
\]
\item ``Two lines meet in at most one point'':
\[
|I_1 \cap I_2|>1 \text{ for some $I_1,I_2 \in \calI$} \implies I \in \mathcal I \text{ for every }3\text{-subset of } I_1 \cup I_2. 
\]
\item ``Two conics meet in at most four points'':
$$|J_1 \cap J_2|>4 \text{ for some $J_1,J_2 \in \calJ$} \implies J \in \mathcal J \text{ for every }6\text{-subset of } J_1 \cup J_2.$$
\end{enumerate}
\end{definition}

\begin{example}
\label{ex:weakcriteria1}
The pair $\mathcal Q=(\{123,145\},\{234678\})$ satisfies all of B\'ezout's weak criteria. The pair $\mathcal Q'=(\{123,234\},\emptyset)$ satisfies the first and (vacuously) the third of B\'ezout's weak criteria, but not the second. The pair $\mathcal Q''=(\{123\},\{123456\})$ does not satisfy the first of B\'ezout's weak criteria. By adding additional triples to $\mathcal Q'$, one may obtain the pair $(\{123,124,134,234\},\emptyset)$ which does satisfy all of B\'ezout's weak criteria. This procedure is impossible for pairs such as $\mathcal Q''$ which fail B\'ezout's first weak criterion. 
\end{example}

\begin{lemma}
\label{lem:necessaryforquatroid}
If $\mathcal Q$ is representable, then $\mathcal Q$ is matroidal and satisfies B\'ezout's weak criteria.
\end{lemma}
\begin{proof}
A representable quatroid is, by definition, matroidal. The proof that a representable quatroid satisfies B\'ezout's weak criteria is suggested by the quotations preceding each criterion in \autoref{def:WeakCriteria} as detailed below. 

Suppose $\p\in\mathcal P$ represents $\mathcal Q$. For every $I \in \calI$, $J\in\calJ$, the line $\ell_I$ and the conic $q_J$ meet in at most two points of $\bfp$ by B\'ezout's Theorem. Therefore, $\calQ$ satisfies the first weak criterion. Similarly, if $I_1,I_2 \in \calI$, then $\ell_{I_1},\ell_{I_2}$ meet in at most one point of $\bfp$, unless $\ell_{I_1} = \ell_{I_2}$. If $\ell_{I_1} = \ell_{I_2}$, then any three points of $\bfp$ lying on such line is indexed by a triple in $\calI$. This shows that $\calQ$ must satisfy the second weak criterion. The proof for the third criterion is similar.
\end{proof}

Motivated by our goal of computing all quatroids, and inspired by B\'ezout's weak criteria, we introduce the following operation on pairs. 
\begin{definition}
Let $\calI$ be a set of triples of $\{1,2,\ldots,8\}$. Define $\mydef{\overline{\calI}}$ to be the smallest set of triples containing $\calI$ which satisfies B\'ezout's second weak criterion. 
Similarly, for a set $\calJ$ of sextuples of $\{1,2,\ldots,8\}$, define $\mydef{\overline{\calJ}}$ to be the smallest set of sextuples containing $\calJ$ which satisfies B\'ezout's third weak criterion.
For a pair $\mathcal Q=(\calI,\calJ)$, define $\mydef{\overline{\mathcal Q}}=(\overline{\calI},\overline{\calJ})$. 
\end{definition}

\begin{lemma}
The sets $\overline \calI$ and $\overline \calJ$ are well-defined, and thus so is $\overline \calQ$.
\end{lemma}
\begin{proof}
Let $\calI$ be a set of triples of $\{1,2,\ldots,8\}$. The definition of $\overline{\calI}$ is the smallest set of triples containing $\calI$ which satisfies B\'ezout's second weak criterion. This set is not empty since the set of \emph{all} triples of $\{1,2,\ldots,8\}$ contains $\calI$ and satisfies the second criterion. Moreover, if $\calI_1$ and $\calI_2$ both satisfy B\'ezout's second weak criterion, then so does their intersection, proving that the smallest such set of triples is well-defined. The argument for sextuples is exactly the same. 
\end{proof}

\begin{example}
Consider the pair $\mathcal Q=(\{123,124,345\},\{245678\})$. Then
\[
\overline{\mathcal Q} = (\{123,124,125,134,135,145,234,235,245,345\},\{245678\}).
\]
Note that $\mathcal Q$ satisfies the first and third weak criteria, but not the second. On the other hand $\overline{\mathcal Q}$ satisfies the second and third weak criteria, but not the first.
\end{example}

When a configuration $\bfp$ is contained in an irreducible cubic, B\'ezout's Theorem applied to such a cubic restricts which lines and conics pass through subsets of $\bfp$. This motivates the following definition.
\begin{definition}
\label{def:Bezoutian}
A pair $\mathcal Q=(\mathcal I,\mathcal J)$ satisfies \mydef{B\'ezout's strong criteria} if 
\begin{enumerate}
\item ``An irreducible cubic meets a line in at most three points'':
$$I_1,I_2 \in \mathcal I \implies |I_1 \cap I_2|\leq 1.$$
\item ``An irreducible cubic meets a conic in at most six points'':
$$J_1,J_2 \in \mathcal J \implies |J_1 \cap J_2| \leq 4.$$
\end{enumerate}
Pairs which satisfy both B\'ezout's weak and strong criteria are called \mydef{B\'ezoutian}. 
\end{definition}

As with the weak criteria, the quotes suggest the proof of an important result.

\begin{lemma}
\label{lem:havingcubicsimpliesbezoutian}
A configuration $\p \in \mathcal P$ satisfies the conditions of  \autoref{prop:FibreIsFinite} if and only if $\p$ represents a B\'ezoutian quatroid. Hence, there are no rational cubics through a configuration $\p$ that represents a non-B\'ezoutian quatroid.
\end{lemma}
\begin{proof}
Let $\calQ$ denote the quatroid represented by $\bfp \in \mathcal P$.   By \autoref{lem:necessaryforquatroid}, $\calQ$ satisfies B\'ezout's weak criteria. Suppose $\p$ satisfies the conditions of  \autoref{prop:FibreIsFinite}. Then there is an irreducible cubic through $\p$ and so no four points of $\p$ lie on a line and no seven lie on a conic. As suggested by the quotations in \autoref{def:Bezoutian}, this implies the strong criteria are satisfied: if $| I_1 \cap I_2| > 1$ for some $I_1,I_2$ then at least four points of $\bfp$ lie on the line $\ell_{I_1} = \ell_{I_2}$ which is a contradiction. The proof for the second criterion is similar. 

Conversely, suppose $\p$ represents a B\'ezoutian quatroid $\calQ=(\calI,\calJ)$. Then no four points of $\p$ lie on a line and no seven lie on a conic and so $\p$ satisfies the conditions in \autoref{prop:FibreIsFinite}. To see this, suppose towards a contradiction that $p_1 \vvirg p_4$ lie on a line: then $123,124,134,234 \in \calI$ so B\'ezout's first strong criterion is violated. A similar conclusion holds if seven points lie on a conic.
\end{proof}

\subsection{Computing all candidate quatroids} 
We approach the task of computing all  quatroids by first computing a list $\mydef{\mathfrak Q}$ of \emph{candidate quatroids}. A pair $\mathcal Q$ is a \mydef{candidate quatroid} if it is matroidal and satisfies B\'ezout's weak criteria. By \autoref{lem:necessaryforquatroid}, all (representable) quatroids are in $\mathfrak Q$.
\begin{remark}
Our goal is to count rational cubics through $\bfp \in \calP$, so by \autoref{lem:havingcubicsimpliesbezoutian} it is enough to determine all B\'ezoutian quatroids. However, for completeness, we compute all quatroids.
\end{remark}
Since matroidal pairs have underlying representable matroids, we take advantage of existing databases of matroids to begin our computation. We use the database \cite{MatroidDB}, to compile a list $\mydef{\mathfrak M}$ of matroids representable by eight distinct points in $\mathbb{P}_{\mathbb{C}}^2$. These are the representable simple matroids of rank at most three. There are $67$ orbits of such matroids, under the action of $\mathfrak S_8$: $66$ of rank three and $1$ of rank two. In this work, we record matroids in terms of their nonbases of size three. The original list from \cite{MatroidDB} is in the auxiliary file \texttt{SimpleMatroids38.txt}.

We now give an algorithm which produces all candidate quatroids with a given underlying matroid by greedily extending its the conic conditions.
\begin{algorithm}[]
\caption{AllConicExtensions}\label{algo:allConicalExtensions}
\SetKwIF{If}{ElseIf}{Else}{if}{then}{elif}{else}{}%
\DontPrintSemicolon
\SetKwProg{All Conical Extensions}{All Conical Extensions}{}{}
\LinesNotNumbered
\KwIn{A pair $\mathcal Q=(\mathcal I, \emptyset)$ where ${\mathcal I} \in \mathfrak M$
 }
\KwOut{A list of all pairs of the form $\mathcal Q' = (\mathcal I,\mathcal J)$ which satisfy B\'ezout's weak criteria.} 
\nl \textbf{initialize }\texttt{PairsFound}$= \{ \mathcal Q \}$\;
\nl \textbf{initialize }\texttt{PairsToExtend}$= \{ \mathcal Q \}$\;
\nl (Compute the stabilizer $H$ of $\mathcal I$ in $\mathfrak S_8$)\;
\nl \While{\texttt{PairsToExtend }$\neq \emptyset$}{
\nl Choose $\mathcal Q' = (\mathcal I,\mathcal J)$ from \texttt{PairsToExtend} and delete $\mathcal Q'$ from $\texttt{PairsToExtend}$\;
\nl Set $\texttt{Q'\_Extensions} = \{\overline{(\mathcal I, \mathcal J \cup \{J\})}\}_{J \text{ a }6\text{-subset of }\{1,\ldots,8\}}$\;
\nl \For{$\mathcal Q'' \in \texttt{Q'\_Extensions}$}{
\nl \If{$\mathcal Q''$ satisfies B\'ezout's weak criteria}
{ 
\nl (Let $\mathcal Q^*$ be a canonical representative of the $H$-orbit of $\mathcal Q''$)\;
\nl \If{$\mathcal Q^* \not\in $\texttt{QuatroidsFound}}{
\nl \texttt{QuatroidsFound}$\leftarrow \mathcal Q^*$\;
\nl \texttt{PairsToExtend}$\leftarrow \mathcal Q^*$\;
}
}}
}
\nl \Return{} \texttt{PairsFound}\;
\end{algorithm} 
\renewcommand{\algocfautorefname}{Algorithm}
\begin{theorem}\label{thm:AlgorithmWorks}
The set of pairs produced by \autoref{algo:allConicalExtensions} on the input $(\calI,\emptyset)$ coincides with the set of candidate quatroids whose underlying matroid is $\calM_\calI$. Thus, applying \autoref{algo:allConicalExtensions} to every matroid in $\mathfrak M$ produces $\mathfrak Q$.
\end{theorem}
\begin{proof}
Clearly all pairs in the output of \autoref{algo:allConicalExtensions} on the input $(\calI, \emptyset)$ are candidate quatroids whose underlying matroid is $\calM_\calI$. We need to show that all candidate quatroids arise in this way. We ignore the parenthetical steps of \autoref{algo:allConicalExtensions}, which are used only to speed up the process by working modulo the $\mathfrak S_8$-action.

By contradiction, assume $\mathcal Q=(\mathcal I, \mathcal J)$ is a candidate quatroid which is \emph{not} in the output of \autoref{algo:allConicalExtensions} applied to $(\mathcal I,\emptyset)$. Let $\mathcal Q'=(\mathcal I,\mathcal J')$ be a maximal candidate quatroid which is generated by \autoref{algo:allConicalExtensions} with $\mathcal J' \subsetneq \mathcal J$. Let $J \in \mathcal J \setminus \mathcal J'$. 

The first time $\calQ'$ is found in \autoref{algo:allConicalExtensions}, it is placed in \texttt{PairsToExtend} at \texttt{step 12}. Therefore, it is then chosen in \texttt{step 5} and the pair $ \calQ'' = (\calI, \bar{\calJ' \cup \{ J \}})$ is one of the elements of \texttt{Q'\_Extensions} at \texttt{step 6}. When $\calQ''$ is chosen in the \textbf{for} loop at \texttt{step 7}, it satisfies the \textbf{if} condition at \texttt{step 8}, because $\bar{\calJ' \cup \{ J \}} \subseteq \calJ$ guarantees $\calQ''$ satisfies B\'ezout's weak criteria. Hence, $\calQ''$ is a candidate quatroid and it is produced by \autoref{algo:allConicalExtensions}. This contradicts the maximality of $\calQ'$.
\end{proof}

\begin{theorem}
\label{thm:candidatequatroids}
The list $\mathfrak Q$ of candidate quatroids consists of $780617$ pairs. Up to the symmetry of $\mathfrak S_8$, these occur in $126$ distinct orbits. The numbers of orbits of each size are tallied below:

\begin{center}
\begin{tabular}{|c||c|c|c|c|c|c|c|c|c|c|}
\hline \hline 
Orbit Size & 1 & 8 & 28 & 35 & 56 & 70 & 105 & 168 & 210  & 280     \\ \hline 
\# Orbits  & 3 & 2 & 2 & 1 & 3 & 1 & 1 & 3 & 2 &  3    \\ \hline \hline 
Orbit Size & 420&  560 & 840 & 1680 & 2520 & 3360 & 5040 & 6720 & 10080 & 20160   \\ \hline
\# Orbits &  2 & 1 & 13 & 4 & 10 & 13 & 17 & 6 & 22 & 17     \\ \hline \hline 
\end{tabular}
\end{center} 
\end{theorem}
\begin{proof}
The candidate quatroids are computed by calling  \autoref{algo:allConicalExtensions} on all representable matroids $\mathfrak M$. This computation may be performed by calling the function \texttt{AllConicExtensions(M)} in the \texttt{Quatroids.jl} package (see \autoref{tab:software} for details). All candidate quatroids are produced by the function \texttt{GenerateAllCandidateQuatroids()} and the results are automatically stored in a text file. The function \texttt{OrbitSizes()} then produces a text file listing the corresponding orbit sizes. These functions use features from \texttt{GAP} \cite{GAP} and \texttt{OSCAR} \cite{OSCAR} to work modulo the $\mathfrak S_8$-symmetry.
\end{proof}

\subsection{Representability of Quatroids}
A candidate quatroid $\calQ \in \frakQ$ is a quatroid if and only if it is representable. In order to compile a list of quatroids, for every candidate quatroid $\calQ$, we either exhibit an element $\bfp \in (\bbP_{\mathbb{C}}^2)^8$ representing it, or we prove that the realization space $\calS_\calQ$ is empty.
\begin{theorem}\label{thm:representability}
Every $\calQ \in \frakQ$ is representable over $\bbQ$ except for the ones in the $\frakS_8$-orbits of 
\begin{align*}
\mathcal Q_{41} &= (\{123,145,167,246,258,357,368,478\},\{\}), \\
\mathcal Q_{63} &= (\{123,145,246,356\},\{125678,134678,234578\}).
\end{align*}
Those in the orbit of $\mathcal Q_{41}$ are representable over $\mathbb{C}$ but not over $\mathbb{R}$. Those in the orbit of $\mathcal Q_{63}$ are not representable over $\mathbb{C}$.
\end{theorem}
\begin{proof}
An explicit $\mathbb{Q}$-realization is given in the auxiliary files for each of the $124$ orbits different from the ones of $\calQ_{41}$ and $\calQ_{63}$. A \texttt{Macaulay2} \cite{M2} script verifying that they are representatives is provided as well. The claims regarding $\calQ_{41}$ and $\calQ_{63}$ are proved in \autoref{lem:Q41} and \autoref{lem:Q63}.
\end{proof}
The following corollary combines \autoref{thm:candidatequatroids} and \autoref{thm:representability}.
\begin{corollary}
\label{cor:orbitcorollary}
Of the $780617$ pairs ($126$ orbits) in $\mathfrak Q$, exactly $779777$ ($125$ orbits) are representable over $\mathbb{C}$ and $778937$ ($124$ orbits) are representable over $\mathbb{R}$ and $\mathbb{Q}$. In particular, the stratification $\mathcal S$ of $\mathcal P$ consists of $779777$ strata, $544748$ ($76$ orbits) of which are B\'ezoutain. 
\end{corollary}

To complete the proof of \autoref{thm:representability}, we show that $\calQ_{41}$ is representable over $\bbC$ and not over $\bbR$, and that $\calQ_{63}$ is not representable. To analyze $\calS_{\calQ_{41}}$ we extend a standard matroid procedure to put coordinates on the realization space of a quatroid.

We represent an element $\bfp \in (\bbP_{\mathbb{C}}^2)^8$ as a $3 \times 8$ matrix, whose columns correspond to elements of $\bbP_{\mathbb{C}}^2$. We are free to work modulo the action of $\PGL_2$ on $\bbP_{\mathbb{C}}^2$; it is convenient to normalize either some of the points of $\bfp$ or some of the conic conditions of $\calQ$, leaving some free parameters $\mydef{z_1,\ldots,z_r}$. These parameters are free to vary in a quasi-projective variety $\mydef{S_{\mathcal Q}} \subseteq \bbC^r$ described by equations and inequations. The equations are determinantal relations imposed by $\mathcal Q$, some of which are identically satisfied after the normalization. The inequations are the determinantal relations \emph{not} imposed by $\mathcal Q$. The realization space $\calS_\calQ$ is a bundle over $S_\calQ$. In particular, $\calQ$ is representable if and only if $S_\calQ$ is nonempty and $\calS_\calQ$ is irreducible if and only if $S_\calQ$ is. 

Candidate quatroid $\mathcal Q_{41}$ is a well-known matroid called the \mydef{MacLane matroid} (see, e.g., \cite{Ingleton,Ziegler:Orientable}) which is realized by a \mydef{M\"obius-Kantor configuration} \cite[Section 2]{MR0038078}. As a matroid, it is non-orientable, not representable over $\mathbb{R}$, and its realization space is reducible. We give a standard proof regarding its representability. 

\begin{lemma}\label{lem:Q41}
The candidate quatroid $\mathcal Q_{41} = (\{123,145,167,246,258,357,368,478\},\emptyset)$ is representable over $\mathbb{C}$ but not $\mathbb{R}$. Its realization space is the union of two distinct orbits of $\PGL_2$.
\end{lemma} 
\begin{proof}
The four points $p_1,p_2,p_4,p_7$ are in general linear position, in the sense that no subset of three of them lie on a line. Therefore, we may normalize them with the action of $\PGL_2$ to be the four points $[e_1],[e_2],[e_3]$ and $[e_1+e_2+e_3]$ in $\bbP_{\mathbb{C}}^2$. Using the fact that $123 \in I$, we deduce $p_3 = [1:z_1:0]$ for some $z_1$ and similarly, we write $p_5,p_6,$ and $p_8$ using the linear relations $145,246,258 \in I$, respectively, and in this order. We obtain the representation of $\bfp$ as the matrix
\[A_{41}=
\begin{bmatrix}
1 & 0 & 1   & 0 & 1   & 0  & 1      & 1\\
0 & 1 & z_1 & 0 & 0   & 1  & 1    & z_4\\
0 & 0 & 0   & 1 & z_2 & z_3& 1 & z_2
\end{bmatrix}
\]
subject to the determinantal equations imposed by $167,357,368,478 \in I$. Conditions $167$ and $478$ imply that $z_3=1=z_4$. Condition $368$ implies that $z_1=1-z_2$. Finally, $357$ implies that the last remaining parameter, $z_2$, satisfies the univariate quadratic equation $-z_2^2+z_2=1$, which has two distinct non-real solutions. For either solution, the inequations hold. Hence the set $S_{\calQ_{41}}$ consists of two points. Since the two points are inequivalent for the action of $\PGL_2$, the two corresponding $\PGL_2$-orbits give rise to two disjoint irreducible components of $\mathcal S_{\mathcal Q_{41}}$.
\end{proof}

We now change focus to the non-representability of $\calQ_{63}$. Our proof relies on the following property: for any realization $\bfp$ of $\calQ_{63}$, the base locus $Z(L_\bfp)$ is reduced. Hence, the matroid underlying the base locus is determined by $\calQ_{63}$ (see \autoref{lem:ReducedBaseLociConditions}) and if such a matroid is non-representable, then $\calQ_{63}$ is non-representable as well. 

\emph{A priori} the Cayley--Bacharach point of a configuration $\p$ of cubics may coincide with one of the eight points in $\p$. In this case the variety underlying the base locus $Z(L_\bfp)$ consists of eight points, but as a scheme, it contains a nonreduced component of degree two supported at the Cayley--Bacharach point. See, for example, \autoref{fig:Nonreduced}. 
\begin{figure}[!htpb]
\includegraphics[scale=0.3]{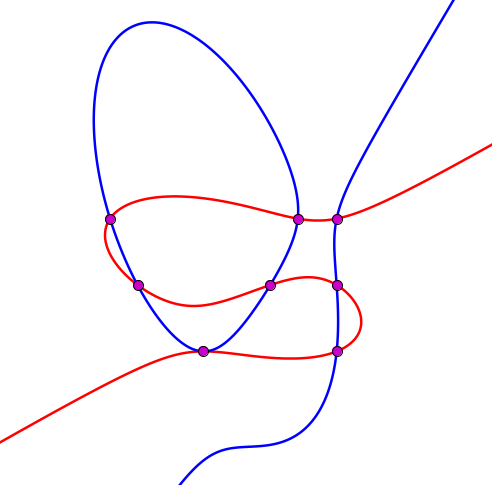}
\caption{A configuration $\p \in \mathcal P$ and two cubics $C_1,C_2 \in L_{\p}$. The base locus $Z(L_{\p})$ is nonreduced, witnessed by the fact that $C_1$ and $C_2$ are tangent at one of the eight points. }
\label{fig:Nonreduced}
\end{figure}

\begin{lemma}
\label{lem:ReducedBaseLociConditions}
Let $\calQ = (\calI,\calJ)$ be a B\'ezoutian candidate quatroid and let $\p \in \mathcal S_{\mathcal Q}$. If $\mathcal Q \not\leq \mathcal Q' $ for every $\mathcal Q'$ in the orbit of $\mathcal Q_{10}$, then the Cayley--Bacharach point of $L_{\p}$ cannot coincide with any point $p_i$ involved in a line condition $I \in \mathcal I$ or conic condition $J \in \mathcal J$.
\end{lemma}
\begin{proof}
Without loss of generality, assume $\mathcal Q'=\mathcal Q_{10}$. Since $\mathcal Q$ is B\'ezoutian, the base locus $Z(L_{\p})$ is $0$-dimensional by \autoref{prop:FibreIsFinite} and \autoref{lem:havingcubicsimpliesbezoutian}. Let $p_9$ be the Cayley--Bacharach point of $\bfp$. By \autoref{rem:OrderIsCombinatorial}, we have that $\calQ_{77} \leq \calQ_{10}$, and so we prove that if $p_9=p_i$ for some $p_i$ involved in a triple or sextuple of $\calQ$, then either $\calQ \leq \calQ_{77}$ or $\calQ \leq \calQ_{10}$. 

Suppose  $p_9$ coincides with $p_i$ for some $i$ involved in a conic condition $J \in \mathcal J$. Without loss of generality, $i = 1$ and $J=123456$. Then $C = q_J \ell_{78}$ is a reducible cubic containing $\bfp$, hence $C \in L_\bfp$. By B\'ezout's Theorem, the conic $q_{J}$ cannot contain the five points $p_2,\ldots,p_6$ and a scheme of length two supported on $p_1$: if this was the case, the intersection of $q_J$ with an irreducible cubic in $L_\bfp$ would have degree (at least) $7$. Hence, $\ell_{78}$ contains $p_1$. Therefore $\mathcal Q \leq \mathcal Q_{10}$. 

Now, suppose that $p_9$ coincides with $p_i$ for some $i$ in a line condition $I \in \mathcal I$. Suppose $i \notin J$ for any $J \in \mathcal J$; if this is not the case, then $\calQ \leq \calQ_{10}$ as in the previous case. Without loss of generality $i =1$ and $I=123$. Then $C = \ell_{123}q_{45678}$ is a reducible cubic containing $\bfp$, hence $C \in L_\bfp$. Similarly to the previous case, by B\'ezout's Theorem, the line $\ell_{123}$ does not contain the subscheme of $Z(L_{\p})$ of length two supported at $p_1$. Hence, the quadric $q_{45678}$ contains $p_1$. Since $\mathcal Q$ does not contain a conic condition involving $i=1$, the quadric $q_{145678}$ must be the union of two lines. Up to reordering the points, those two lines are $\ell_{145}$ and $\ell_{678}$ implying that $\mathcal Q \leq \mathcal Q_{77}$. 
\end{proof}

\begin{corollary}\label{cor:ConicLineSwap}
If the base locus $p_1,\ldots,p_9$ of $L_{\p}$ is a reduced $0$-dimensional scheme, then for any sextuple $J \subset \{1,2,\ldots,9\}$, the points $\{p_j\}_{j \in J}$ lie on a quadric if and only if $\{p_i\}_{i \not\in J}$ lie on a line. 
\end{corollary}

A candidate quatroid $\mathcal Q = (\calI, \calJ)$ is called \mydef{exhaustive} if every $i \in \{1,2,\ldots,8\}$ is involved in some $I \in \mathcal I$ or $J \in \mathcal J$. \autoref{lem:ReducedBaseLociConditions} has the following strong consequence.
\begin{corollary}
\label{cor:ExhaustiveReduced}
Let $\mathcal Q$ be an exhaustive B\'ezoutian quatroid such that $\calQ \not\leq \calQ'$ for every $\mathcal Q'$ in the orbit of $\mathcal Q_{10}$. Then for every $\p \in \mathcal S_{\mathcal Q}$, the base locus $Z(L_{\p})$ is reduced.
\end{corollary}

\autoref{cor:ExhaustiveReduced} says that there exist quatroids for which every realization gives a reduced base locus. We say that these quatroids, themselves,  \mydef{have reduced base locus}. In this case, knowing the matroid underlying the reduced base locus of the pencil of cubics is equivalent to knowing the quatroid underlying any $8$-subset of the base locus by \autoref{cor:ConicLineSwap}.

\begin{lemma}\label{lem:Q63}
The pair $
\mathcal Q_{63} = (\{123,145,246,356\},\{125678,134678,234578\})$ is not $\bbC$-representable. 
\end{lemma}
\begin{proof}
Suppose towards a contradiction that $\bfp$ represents $\mathcal Q_{63}$. Note that $\mathcal Q_{63}$ is B\'ezoutian, exhaustive, and $\mathcal Q_{63} \not\leq \calQ'$ for every $\calQ'$ in the orbit of $\mathcal Q_{10}$. By \autoref{cor:ConicLineSwap}, $\mathcal Q_{63}$ has reduced base locus and the matroid underlying $Z(L_{\p})$, as represented by nonbases, is
\[
M_{Z(L_\p)} = \{123,145,246,356,349,259,169\}.
\]
This is the Fano matroid, which is famously not $\mathbb{C}$-representable \cite[Prop.~6.4.8]{Oxley}. Hence $\mathcal Q_{63}$ is not $\mathbb{C}$-representable.
\end{proof}
\noindent Inspired by the proof of \autoref{lem:Q63}, we name $\mathcal Q_{63}$ the \mydef{Fano candidate}.

\section{Irreducibility of quatroid strata}\label{sec:IrreducibilityOfStrata}
In this section, we establish which realization spaces of quatroids are irreducible. 
\begin{theorem}\label{thm:allirreducible}
 For every quatroid $\calQ$ except those in the orbit of $\calQ_{41}$, the realization space $\calS_\calQ$ is irreducible. If $\calQ$ is in the orbit of $\calQ_{41}$, then $\calS_{\calQ}$ is the union of two irreducible components.
\end{theorem}
The statement for $\calQ_{41}$ has been shown in \autoref{lem:Q41}. The rest of the proof of \autoref{thm:allirreducible} is built on a series of reductions, inspired by \cite[Thm.\ 4.5]{ClGrMoMo:MatroidStratification}. To state them precisely, we introduce a definition. 

\begin{definition}
Let $\mathcal Q=(\mathcal I,\mathcal J)$ be a quatroid and let $i \in \{1,2,\ldots,8\}$. The \mydef{deletion of $i$ from $\mathcal Q$} is
\[
\mydef{\mathcal Q-\{i\}}=(\{I \in \mathcal I \mid i \notin I\},\{J \in \mathcal J \mid j \notin J\}).
\]
\end{definition}
\begin{proposition}
\label{prop:FewFigureReduction}
Let $\calQ = (\calI,\calJ)$ be a quatroid and $i\in\{1,2,\ldots,8\}$. Suppose one of the following conditions holds:\begin{itemize}
\item $i$ is in at most two elements $I_1,I_2$ of $\mathcal I$ and no elements of $\mathcal J$,
\item $i$ is in at most one element of $\mathcal J$ and no elements of $\mathcal I$.
\end{itemize}
Then $\calS_{\calQ}$ is irreducible if and only if $\calS_{\calQ - \{i\}}$ is irreducible. In particular, if $\mathcal Q = (\mathcal I, \emptyset)$ and $|\mathcal I|\leq 6$ then $\mathcal S_{\mathcal Q}$ is irreducible. 
\end{proposition}
\begin{proof}
First, notice that if $i$ appears in no elements of $\calI$ and $\calJ$, then $\calQ = \calQ - \{i\}$. Hence, we may assume $i$ appears in at least one element of $\calI \cup \calJ$. Without loss of generality, let $i = 8$.

Suppose first $i$ belongs to exactly one element of $I \in \calI$, say $I = 678$. Let $\pi: (\bbP_{\mathbb{C}}^2)^8 \to (\bbP_{\mathbb{C}}^2)^7$ be the projection on the first seven factors. Let $X = \pi (\calS_\calQ)$. First, observe $\calS_{\calQ - \{i\}}$ is a dense subset of $X \times \bbP_{\mathbb{C}}^2$: indeed, the line and conic conditions of $\calQ - \{i\}$ are the line and conic conditions of $\calQ$ not involving $i$, so they are satisfied by $\bfp \in (\bbP_{\mathbb{C}}^2)^8$ if and only if $\pi(\bfp) \in X$. In particular, $X$ is irreducible if and only if $\calS_{\calQ - \{i\}}$ is irreducible. Moreover, $\calS_\calQ$ is an open dense subset of a $\bbP_{\mathbb{C}}^1$-bundle over $X$, embedded in  $X \times \bbP_{\mathbb{C}}^2$: the fiber over $\bfp' \in X$ is an open subset of the line $\ell_{67}$. In particular, $X$ is irreducible if and only if $\calS_{\calQ}$ is irreducible. 

If $i$ belongs to exactly one element $J \in \calJ$, the proof is similar. Suppose $J = 345678$. The proof follows the same argument described above, with the only difference that $\calS_{\calQ}$ is an open subset of a fiber bundle over $X$ whose fiber at $\bfp$ is the conic $q_{34567}$. 

If $i$ belongs to two exactly two elements $I_1,I_2 \in \calI$, suppose without loss of generality $I_1 = 678$ and $I_2 = 458$. Again, the proof is similar. In this case $\calS_\calQ$ is birational to $X$, with a birational map given by $\phi\colon X \to \calS_{\calQ}$ defined by $\phi(\bfp') = (\bfp', p_8)$ where $p_8 = \ell_{45} \cap \ell_{78}$. 
 \end{proof}
 
 A second reduction is built on the Cayley--Bacharach construction introduced in \autoref{sec:Quatroids}. Let $\mathcal Q$ be an exhaustive B\'ezoutian quatroid such that $\mathcal Q \not\leq \mathcal Q'$ for every $\calQ'$ in the orbit of $\calQ_{10}$, and let $\bfp$ be any of its realizations. Then $Z(L_\bfp)$ is reduced by \autoref{cor:ExhaustiveReduced}. The underlying matroid of $Z(L_{\bfp})$ is well-defined by \autoref{cor:ConicLineSwap} and so is the set of quatroids obtained by deleting any of the nine points. If $\calQ'$ is one such quatroid, write $\mathcal Q \mydef{\leadsto} \mathcal Q'$ and say that $\mathcal Q'$ is a \mydef{modification} of $\mathcal Q$. 

\begin{lemma}\label{lem:birational realization}
Let $\calQ, \calQ'$ be quatroids with $\mathcal Q \leadsto \mathcal Q'$. If  $\mathcal S_{\mathcal Q'}$ is irreducible then so is $\mathcal S_{\mathcal Q}$. 
\end{lemma}
\begin{proof}
Let $\mathcal Q \leadsto \mathcal Q'$ and write $\mathcal M$ for the matroid of the base locus $Z(L_{\p})$ for any $\p \in \mathcal S_{\mathcal Q}$. Consider the realization space $\mathcal S_{\mathcal M} \subseteq (\mathbb{P}_{\mathbb{C}}^2)^9$ of that matroid. The variety $\mathcal S_{\mathcal Q}$ is isomorphic to the variety $\mathcal K \subseteq \mathcal S_{\mathcal M} \subseteq (\mathbb{P}_{\mathbb{C}}^2)^9$ of configurations which are complete intersections of type $(3,3)$. Since $\mathcal S_{\mathcal Q'}$ is assumed to be irreducible, the map from $\mathcal K$ to $\mathcal S_{\mathcal Q'}$ given by forgetting one of the nine points shows that $\mathcal K$ is birational to $\mathcal S_{\mathcal Q'}$. 
\end{proof}

Based on these results, given a quatroid $\mathcal Q$, there are four ways to reduce the problem of determining irreducibility of $\mathcal S_{\mathcal Q}$: 
\begin{enumerate}[(i)]
 \item $\calQ$ reduces to $\calQ' = \calQ - \{i\}$ with $i$ involved in at most two line conditions and no conic condition,
 \item $\calQ$ reduces to $\calQ' = \calQ - \{i\}$ with $i$ involved in no line condition and one conic condition,
 \item $\calQ$ has at most six line conditions and no conic condition,
 \item $\calQ \leadsto \calQ'$.
\end{enumerate}
In these cases, if $\calS_{\calQ'}$ is irreducible, so is $\calS_{\calQ}$.
\begin{lemma}
\label{lem:ABCreductions}
All quatroids in orbits other than the $24$ orbits represented by
\begin{gather*} \mathcal Q_{4},
 \mathcal Q_{5},
 \mathcal Q_{13},
 \mathcal Q_{19},
 \mathcal Q_{21},
 \mathcal Q_{22},
 \mathcal Q_{25},
 \mathcal Q_{38},
 \mathcal Q_{41},
 \mathcal Q_{46},
 \mathcal Q_{49},
 \mathcal Q_{51}, \\
 \mathcal Q_{53},
 \mathcal Q_{57},
 \mathcal Q_{58},
 \mathcal Q_{59},
 \mathcal Q_{62},
 \mathcal Q_{66},
 \mathcal Q_{73},
 \mathcal Q_{76},
 \mathcal Q_{81},
 \mathcal Q_{82},
 \mathcal Q_{101},
 \mathcal Q_{123}
\end{gather*}
have irreducible realization spaces. 
\end{lemma}
\begin{proof}
The result is obtained by iteratively applying reductions (i)--(iii). The proof is computational and the reductions are performed by the function \texttt{QuatroidReductions()} in the \texttt{Quatroids.jl} package. This function produces the file \texttt{ReductionProofs.txt}.
\end{proof}

\begin{lemma}
\label{lem:Modifications}
The following relationships hold
\begin{gather*}
\mathcal Q_{4}\leadsto\mathcal Q_{17}, \quad \,\, \,\,\,
\mathcal Q_{5} \leadsto \mathcal Q_{25}, \quad \,\, \,\, \mathcal Q_{13}\leadsto\mathcal Q_{20},\quad \,\,
 \mathcal Q_{19}\leadsto\mathcal Q_{60},\quad \,\,
 \mathcal Q_{22}\leadsto\mathcal Q_{30},\quad \,\,
 \mathcal Q_{46}\leadsto\mathcal Q_{25},\quad \,\,\\
 \mathcal Q_{49}\leadsto\mathcal Q_{25},\quad \,\,
 \mathcal Q_{51}\leadsto\mathcal Q_{25},\quad \,\, 
 \mathcal Q_{53}\leadsto\mathcal Q_{40}, \quad \,\,
 \mathcal Q_{57}\leadsto\mathcal Q_{30},\quad \,\,
 \mathcal Q_{59}\leadsto\mathcal Q_{32},\quad \,\,
 \mathcal Q_{62}\leadsto\mathcal Q_{32},\quad \,\,\\
 \mathcal Q_{66}\leadsto\mathcal Q_{36},\quad \,\,
 \mathcal Q_{73}\leadsto\mathcal Q_{37},\quad \,\,
 \mathcal Q_{76}\leadsto\mathcal Q_{39},\quad \,\,
 \mathcal Q_{81}\leadsto\mathcal Q_{45},\quad \,\,
 \mathcal Q_{82}\leadsto\mathcal Q_{43}.
 \end{gather*}
 In particular, the realization spaces of the quatroids in the same orbits as
 \[
 \mathcal Q_{4},
 \mathcal Q_{13},
 \mathcal Q_{19},
 \mathcal Q_{22},
 \mathcal Q_{57},
 \mathcal Q_{59},
 \mathcal Q_{62},
 \mathcal Q_{66},
 \mathcal Q_{73},
 \mathcal Q_{81},
 \mathcal Q_{82}
 \]
 are all irreducible.
\end{lemma}
\begin{proof}
All quatroids in the statement satisfy the conditions of \autoref{cor:ExhaustiveReduced}: this is verified computationally, see \autoref{tab:software}. The relation $\calQ \leadsto \calQ'$ is also verified computationally via the function \texttt{Modifications(Q)} in the package \texttt{Quatroids.jl}. Note that each quatroid on the right-hand-side of a $\leadsto$ sign, other than $\mathcal Q_{25}$, has already been shown to have irreducible realization space, therefore \autoref{lem:birational realization} guarantees that the corresponding quatroids on the left-hand-side have irreducible realization spaces. 
\end{proof}
The following result completes the proof of \autoref{thm:allirreducible}.
\begin{theorem}\label{thm: explicit cases irreducible}
The realization spaces of quatroids 
\[
 \mathcal Q_{21},
 \mathcal Q_{25},
 \mathcal Q_{38},
 \mathcal Q_{58},
 \mathcal Q_{101},
 \mathcal Q_{123}
 \]
are irreducible.
\end{theorem}
\begin{proof}
The auxiliary files include \texttt{Macaulay2} scripts to reproduce the computations in this proof.

We normalize elements of the realization space in two different ways. Either we normalize a conic conditions to be $x_0^2 - x_1x_2$ and then three points on such conic to be $[0:1:0],[0:0:1],[1:1:1]$, or we normalize four points in $\bbP^2_\bbC$ to be $[1:0:0],[0:1:0],[0:0:1],[1:1:1]$. After this normalization, a point in the realization space of quatroid $\calQ_i$ is represented by one of the following $3 \times 8$ matrices:
\begin{equation*}
\begin{array}{rlrl}
 A_{21} &= \left[\begin{array}{cccccccc}
0 & 0 & 0 & 1 & 1 & z_3 & z_4 & z_5 \\
1 & 0 & 1 & 1 & z_2 & z_3^2 & z_4^2 & z_5^2 \\
0 & 1 & z_1 & 1 & 1 & 1 & 1 & 1 \end{array}\right],
 &A_{25} &= \left[\begin{array}{cccccccc}
                 1 & 0 & 1 & 0 & 1 & 1 & z_3  &  z_4 \\ 
                 0 & 1 & z_1 & 0 & 0 & 1 & 1  &  z_5 \\ 
                 0 & 0 & 0 & 1 & z_2 & 1 & 1  &  z_6 
                 \end{array}
 \right], \\
~\\
A_{38} &= \left[\begin{array}{cccccccc}
           1&   1&    0&  0& 1&  z_3  & z_4& z_5\\
           1&   z_2&  1&  0& 1&  z_3^2 & z_4^2 & z_5^2\\
           z_1& z_1&  0&  1& 1&  1&     1& 1
                 \end{array}
 \right],
 &A_{58} &= \left[\begin{array}{cccccccc}
    1& 1  &  0& 0& z_1&    1& z_2&   z_3 \\  
    z_5& 1  &  1& 0& z_1^2&  1& z_2^2& z_3^2 \\
    z_4&z_4&  0& 1 & 1&     1& 1  & 1 
    \end{array} \right], \\
~\\
A_{101} &= \left[\begin{array}{cccccccc}
        1&0&1&1&0&1&0&z_5 \\
        0&1&z_1&z_2&0 &0&1&1 \\
        0&0&0&0&1&z_3&z_4&z_4 \\
        \end{array}
         \right], 
 & A_{123} &= \left[\begin{array}{cccccccc}
0 & 0 & 0 &   0   & 1 & z_3 & z_4 & z_5 \\
1 & 0 & 1 &   1   & 1 & z_3^2 & z_4^2 & z_5^2 \\
0 & 1 & z_1 & z_2 & 1 & 1     & 1 & 1 \end{array}\right],
\end{array}
\end{equation*}
where the vectors $(z_1 \vvirg z_r)$ are subject to the determinantal relations imposed by the line and conic condition of a quatroid. 

We verify in the auxiliary files that the additional conditions cut out an irreducible variety. To prove irreducibility, we reduce to a hypersurface case. To prove irreducibility of a hypersurface, we observe that its singular locus has codimension at least three (in the ambient space). 
\end{proof}

That every quatroid stratum other than $\mathcal S_{\mathcal Q_{41}}$ is irreducible guarantees that the notion of generic point makes sense on these strata. For $\mathcal S_{\mathcal Q_{41}}$, one may perform every computation exhaustively, since up to the $\PGL_2$-action, this realization space is just two complex conjugate points.

\section{The generic number of rational cubics through quatroid strata: a lower bound}
\label{sec:lowerbounds}

Recall from \autoref{def:dp} that for $\p \in \mathcal P$ the value $d_{\p}$ is the number of rational cubics through $\p$, counted with multiplicity. The irreducibility of a realization space $\calS_{\calQ}$ guarantees that this number is constant on a dense open subset of $\calS_{\calQ}$. Denote by $\mydef{d_\calQ}$ this generic value. Our goal is to compute $d_\calQ$ for all quatroids. 

In the case of $\mathcal Q_{41}$, it is classically known that, for every $\bfp \in \calS_{\calQ_{41}}$, $d_\bfp = 0$. Indeed, the singular cubics on the pencil $L_\bfp$ are four fully reducible cubics $f_1 \vvirg f_4$ each of the form $\ell_1\ell_2\ell_3$ with $\ell_1, \ell_2, \ell_3$ linearly independent. We refer, for instance, to \cite{ArtDol:HessePencil} for details. 

Let $\mydef{\frakB}$ be the set of B\'ezoutian quatroids. By \autoref{prop:FibreIsFinite}, if $\calQ \notin \frakB$, then $d_\calQ = 0$. For B\'ezoutian quatroids, we obtain a lower bound on $d_\calQ$ via a semicontinuity argument. 

\begin{lemma}
\label{lem:semicontinuityofcardinality}
The function $\p \mapsto d_{\p}$ is lower semicontinuous. 
\end{lemma}
\begin{proof} 
By \autoref{prop:FibreHas12Points}, $L_\bfp \cap \calD$ is either infinite, in which case $d_\bfp = 0$, or a $0$-dimensional scheme of degree $12$. The number $d_{\p}$ counts, with multiplicity, the number of intersection points $L_{\p} \cap \mathcal D$ not occurring on the variety of reducible cubics. Equivalently $r_\bfp = 12 - d_\bfp$ is the degree of the subscheme of $L_\bfp \cap \calD$ supported on the subvariety of reducible cubics. Since this subvariety is closed, upper semicontinuity of $r_\bfp$ follows and we conclude that $d_\bfp$ is lower semicontinuous.
\end{proof}

\begin{corollary}
\label{cor:lowerbounds}
Let $\mathcal Q$ be a quatroid. The number $\max_{\p \in \mathcal S_{\mathcal Q}}(d_{\p})$ is well-defined and equal to $d_{\mathcal Q}$. In particular, for any $\p \in \overline{\mathcal S_{\mathcal Q}}$ we have that $d_\p \leq d_{\mathcal Q}$. Consequently, $d_{\mathcal Q'} \leq d_{\mathcal Q}$ if $\mathcal Q' \leq \mathcal Q$.
\end{corollary}

A consequence of \autoref{cor:lowerbounds} is that any point $\bfp \in \calS_{\calQ}$ provides a lower bound $d_\bfp \leq d_\calQ$. We compute the number $d_\bfp$ for the rational representatives in \texttt{RationalRepresentatives.txt} using a symbolic calculation. We now discuss how to perform that calculation.

\begin{algorithm}[!htpb]
\caption{Symbolic Computation of $d_{\p}$}\label{algo:symbolicrational}
\SetKwIF{If}{ElseIf}{Else}{if}{then}{elif}{else}{}%
\DontPrintSemicolon
\SetKwProg{Number of rational cubics}{Number of rational cubics}{}{}
\LinesNotNumbered
\KwIn{A rational configuration $\p=(p_1,\ldots,p_8) \in \mathcal P$ representing a B\'ezoutian quatroid}
\KwOut{$d_{\p}$ } 
\nl \textbf{initialize }$I$ to be the ideal of the points $p_1,\ldots,p_8 \in \mathcal P$\;
\nl Find a rational basis $\{C_0,C_1\}$ for $I$ in degree $3$\;
\nl Set $F=\textrm{disc}(t_0C_0+t_1C_1) \in \mathbb{Q}[t_0,t_1]$\;
\nl Factor $F$ over $\mathbb{Q}$\\
Optionally, determine the multiplicities of each solution from the powers of the factors\;
\nl Let $R$ be the $k$ linear factors of $F$ with roots $\{[t_0^{(i)}:t_1^{(i)}]\}_{i=1}^k$ \;
\nl For each $i=1,\ldots,k$ factor the cubic $t_0^{(i)}C_0 + t_1^{(i)}C_1$\;
\nl Let $r_\p$ be the number of cubics which factor nontrivially, counted with multiplicity \\
 Optionally, compute the type of each reducible (see \autoref{tab:orbits})\; 
\nl \Return{$12-r_\p$}\;
\end{algorithm} 
\renewcommand{\algocfautorefname}{Algorithm}

\begin{lemma}
\label{lem:reduciblesmustfactoroverQ}
Let $\p \in \mathcal P $ be a rational representative of a B\'ezoutian quatroid. Let $C$ be a reducible cubic through $\p$ defined by a cubic form $f$. Then $f \in \mathbb{P}{S}^3\mathbb{Q}^3$ and it factors over $\mathbb{Q}$. 
\end{lemma}
\begin{proof}
Since $C$ is a reducible cubic, then $C$ is a union of either (a) three lines or (b) a line and a conic. In case (a), since $\p$ represents a B\'ezoutian quatroid, then each line contains (at least) two elements of $\bfp$ by the pigeonhole principle. In particular, since $\bfp$ is rational, each of the three lines has a rational equation. Similarly, in case (b), since $\p$ is B\'ezoutian, the pigeonhole principal implies that there are at least two points on the line and at least five points on the conic. The equation of a conic through five rational points, as well as the equation of a line through two rational points, has rational coefficients. Consequently, in either (a) or (b), the components of $C$ must be realized by rational forms whose product is $f$, up to scaling.
\end{proof}

\autoref{algo:symbolicrational} uses \autoref{lem:reduciblesmustfactoroverQ} to find the reducible cubics, and their multiplicities, on a pencil.

\begin{proposition}
 \autoref{algo:symbolicrational} is correct, in the sense that its output is the number $d_\bfp$. Moreover, the optional steps correctly determine the multiplicities of each cubic through $\p$ and the type of each reducible cubic in the sense of \autoref{tab:orbits}.
\end{proposition}
\begin{proof}
Consider a rational configuration $\p \in \mathcal S_{\mathcal Q}$  for a B\'ezoutian quatroid $\calQ$. By \autoref{prop:FibreIsFinite}, $L_\bfp = \langle C_1,C_2\rangle$ and since $\bfp$ is rational the cubics $C_1,C_2$ can be taken to have rational coefficients. Step $3$ parametrizes the pencil $L_\bfp$ and evaluates the discriminant of plane cubics, $\disc$, on it. By the separability of $\mathbb{Q}$, no irreducible factor has a multiple root, and so the multiplicities of the linear factors over $\mathbb{C}$ are witnessed by the powers of the irreducible factors over $\mathbb{Q}$. This shows that the first optional step correctly detects the multiplicities.

By \autoref{lem:reduciblesmustfactoroverQ} the reducible cubics through rational configurations are themselves defined over $\bbQ$, hence all of them appear in the $\mathbb{Q}$-factorization of $F$ as linear factors. Moreover \autoref{lem:reduciblesmustfactoroverQ} also states that the $\mathbb{C}$-factorization of each such reducible cubic is realized over $\mathbb{Q}$. Consequently, those which factor nontrivially are precisely the reducible cubics on $L_\bfp$, counted with the appropriate multiplicity, and so their number is $r_{\p}$. Since $12=d_\p+r_\p$, the number $d_{\p}$ is correctly determined. The type of each individual reducible cubic may be determined in several elementary ways. In \texttt{countReducibleCubics.m2}, we determine the type via their singular locus.
\end{proof}

\begin{theorem}
\label{thm:lowerbounds}
The following numbers are lower bounds for $d_{\mathcal Q_i}$, where $\mathcal Q_i$ is a B\'ezoutian quatroid:

\begin{center}
{{{\footnotesize{
\begin{tabular}{|r||c|c|c|c|c|c|c|c|c|c|c|c|c|c|c|c|c|c|c|c|} 
\hline
\begin{tabular}{c}  $i\,\,\,$ \end{tabular} & 1 & 2 & 3 & 4 & 5 & 
8 & 9 & 10 & 11 & 12 & 
13 & 15 & 16 & 17 & 18 & 
19 & 20 & 21 & 22 & 24  \\ \hline 
\begin{tabular}{c} {\color{white}{.}}\\ $d_{\mathcal \p_i}$ \\ {\color{white}{.}}\end{tabular} & 12 & 10 & 8 & 6 & 4 & 10 & 8 & 9 & 6 & 7 & 4 & 8 & 6 & 6 & 7 & 4 & 4 & 5 & 2 & 6  \\ \hline \hline 
&25 & 26 & 28 & 29 & 30 & 
31 & 32 & 33 & 34 & 35 & 
36 & 37 & 38 & 39 & 40 & 
41 & 42 & 43 & 44 & 45  \\  \hline 
\begin{tabular}{c}{\color{white}{$d_{\mathcal \p_i}$}} \\ {\color{white}{.}}\end{tabular}&4 & 5 & 4 & 3 & 2 & 1 & 0 & 0 & 2 & 1 & 1 & 3 & 2 & 2 & 1 & 0 & 3 & 3 & 4 & 5 \\ \hline \hline 
&46 & 47 & 48 & 49 & 50 & 
51 & 52 & 53 & 54 & 55 & 
56 & 57 & 58 & 59 & 60 & 
61 & 62  & 64 & 65 &66 \\ \hline 
\begin{tabular}{c}{\color{white}{$d_{\mathcal \p_i}$:}} \\ {\color{white}{.}}\end{tabular}&3 & 4 & 4 & 2 & 4 & 2 & 3 & 1 & 6 & 4 & 5 & 2 & 3 & 0 & 4 & 2 & 0  & 5 & 3 &1\\ \hline \hline  
 & 67 & 68 & 69 & 70 & 
71 & 72 & 73 & 74 & 75 & 
76 & 77 & 78 & 79 & 80 & 
81 & 82 &    &    &     &\\ \hline 
\begin{tabular}{c}{\color{white}{$d_{\mathcal \p_i}$:}} \\ {\color{white}{.}}\end{tabular} & 5 & 6 & 4 & 7 & 5 & 6 & 3 & 6 & 4 & 2 & 8 & 6 & 9 & 7 & 5 & 3 & & & &\\ \hline \hline 
\end{tabular}
}}}}
\end{center}
Moreover, for each rational representative $\p_i$, the irreducible cubics through $\p_i$ all appear with multiplicity one and the reducible cubics are either conic+secant lines or triangles (see \autoref{tab:orbits}).
\end{theorem}
\begin{proof}
For each quatroid $\mathcal Q_i$, with $i \neq 41,63$, we apply \autoref{algo:symbolicrational} to  the $\mathbb{Q}$-representative $\p_i \in \mathcal S_{\mathcal Q_i}$ in \texttt{RationalRepresentatives.txt}. \autoref{cor:lowerbounds} implies the results are lower bounds for $d_{{\calQ_i}}$. The optional steps are taken to determine the types and multiplicities of all cubics. 
\end{proof}

\begin{example}

\begin{figure}[!htpb]
\includegraphics[scale=0.4]{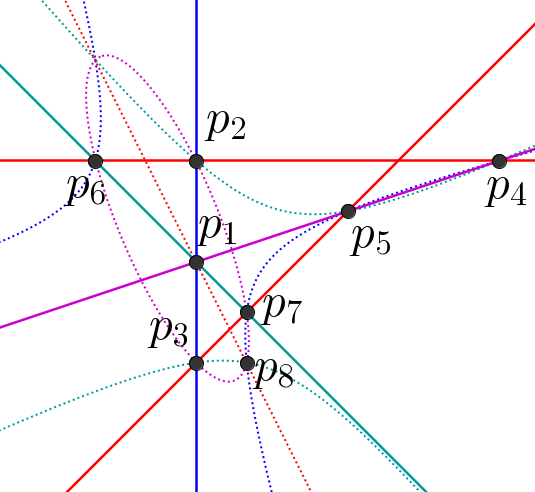}
\caption{The four reducible cubics through the configuration $\p$. The bold figures are the five lines imposed by $\mathcal Q_{43}$. Each set of figures of the same color union to be one of the four reducible cubics through $\p$. }
\label{fig:SymbolicCalculation}
\end{figure}

We detail \autoref{algo:symbolicrational} on the representative 
\[
\p=\left[\!\begin{array}{cccccccc}
      -2&2&2&1&-2&1&2&-2\\
      2&-2&-2&2&-1&-2&-1&1\\
      0&2&-2&1&-1&1&-1&2
      \end{array}\!\right]\]
 of $\mathcal Q_{43}=(\{123,145,167,246,357\},\{\})$. A basis for the ideal of $\bfp$ in degree $3$ is  given by the cubics
 \begin{align*}
 C_0&= 2x_0^2x_1+2x_0x_1^2-2x_0^2x_2-3x_0x_1x_2-2x_1^2x_2+x_0x_2^2+x_1x_2^2+x_2^3\\
 C_1&=8x_0^3+8x_1^3+8x_0^2x_2-3x_0x_1x_2+8x_1^2x_2-29x_0x_2^2-29x_1x_2^2-19x_2^3   \end{align*}
Evaluating the cubic discriminant on the pencil $t_0C_0+t_1C_1$ and factoring it yields
\[
\left(t_{1}\right)^{3}\left(t_{0}-19\,t_{1}\right)^{2}\left(t_{0}-13\,t
      _{1}\right)^{2}\left(t_{0}-5\,t_{1}\right)^{2}\left(375\,t_{0}^{3}-6\,047
      \,t_{0}^{2}t_{1}-25\,539\,t_{0}t_{1}^{2}+232\,083\,t_{1}^{3}\right).
\]
We see that there are at most three reducible cubics which occur with multiplicity $2$ and at most one reducible cubic occurring with multiplicity $3$. Since the nonlinear factor appears with multiplicity one, its $\mathbb{C}$ factors appear with multiplicity one as well.
By factoring $C_{[t_0:t_1]} = t_0C_0+t_1C_1$ for $[t_0:t_1] \in \{[0:1],[19:1],[13:1],[5:1]\}$ we see that all these cubics are reducible:
\begin{align*}
C_{[0:1]} &=\left(x_{1}-x_{2}\right)\left(x_{0}-x_{2}\right)\left(2\,x_{0}+2\,x_{1}+
      x_{2}\right)\\
C_{[19:1]} &= \left(x_{0}+x_{1}\right)\left(4\,x_{0}^{2}+15\,x_{0}x_{1}+4\,x_{1}^{2}-15\,x_{0}x_{2}-15\,x_{1}x_{2}-5\,x_{2}^{2}\right)\\
C_{[13:1]} &= \left(x_{0}+x_{1}-3\,x_{2}\right)\left(4\,x_{0}^{2}+9\,x_{0}x_{1}+4\,x_{1}^{2}+3\,x_{0}x_{2}+3\,x_{1}x_{2}+x_{2}^{2}\right)\\
C_{[5:1]} &= \left(x_{0}+x_{1}+x_{2}\right)\left(4\,x_{0}^{2}+x_{0}x_{1}+4\,x_{1}^{2}-5\,x_{0}x_{2}-5\,x_{1}x_{2}-7\,x_{2}^{2}\right)
\end{align*} The first is the union of three lines with no mutual intersection (i.e.\ a triangle). The remaining three reducible cubics are conic+secant lines. As a result, $d_{\p} = 12 - 3 - 3\cdot 2 = 3$. We display the configuration $\p$ and these reducible cubics in \autoref{fig:SymbolicCalculation}.
\end{example}

\section{The generic number of rational cubics through quatroid strata: an upper bound}\label{sec:expectedintersections}
In this section, we construct upper bounds for each $d_{\mathcal Q}$, and observe that they coincide with the lower bounds listed in \autoref{thm:lowerbounds}. Bounding $d_\mathcal Q$ from above is equivalent to bounding $r_{\p}=12-d_{\p}$ from below, for every $\p \in \mathcal S_{\mathcal Q}$. The value $r_{\p}$ occurs as a sum of intersection multiplicities, which is bounded by a sum of \emph{multiplicities} in the sense of \eqref{eq:multiplicity}:
\begin{equation}
\label{eq:rpversusepsilonp}
\begin{aligned}
r_{\p} = \sum_{\textrm{reducible } C \in L_{\p} \cap \mathcal D} \textrm{imult}_{C}(L_{\p},\mathcal D) &\geq  \sum_{\textrm{reducible } C \in L_{\p} \cap \mathcal D} \textrm{min}_{C \in L \in \textrm{Gr}(2,S^3 \bbC^3)}(\textrm{imult}_C(L,\mathcal D))  \\ &=  \sum_{\textrm{reducible } C \in L_{\p} \cap \mathcal D} \textrm{mult}_{\mathcal D}(C) =: \mydef{m_{\p}}.
\end{aligned}
\end{equation}
Hence, we obtain the bound $d_{\p} =12-r_{\p} \leq  12-m_{\p}$. Define $\mydef{m_{\mathcal Q}}=\min_{\p \in \mathcal S_{\mathcal Q}}(m_{\p})$.
\begin{lemma}
\label{lem:genericepsilonp}
Let $\mathcal Q$ be a B\'ezoutian quatroid. The value $m_{\mathcal Q}$ is well-defined and equal to $m_\p$ for generic $\p \in \mathcal S_{\mathcal Q}$.
\end{lemma}
\begin{proof}
The statement follows from the fact that the multiplicity of a point on a variety is an upper semicontinuous function of the point.
\end{proof}
With these definitions, we can bound $d_\mathcal Q$ from above: 
\begin{equation}
\begin{aligned}
\label{eq:ubfrommult}
d_{\mathcal Q} = \max_{\p \in \mathcal S_{\mathcal Q}}(d_{\p}) = \max_{\p \in \mathcal S_{\mathcal Q}}(12-r_{\p}) =  12 - \min_{\p \in \mathcal S_{\mathcal Q}}(r_{\p})\leq  12 - \min_{\p \in \mathcal S_{\mathcal Q}}(m_{\p}) = 12-m_{\mathcal Q}
\end{aligned}
\end{equation}
The value $12-m_{\mathcal Q}$ may be thought of, for now, as an \emph{expected} number of rational cubics through $\mathcal Q$. 

\subsection{An upper bound for $d_\mathcal Q$ through multiplicities of reducible cubics}
Equation \eqref{eq:ubfrommult} gives us a way to compute upper bounds for $d_\mathcal Q$. Crucially, as shown in \autoref{lem:reducibleCubicsThroughp} a B\'ezoutian quatroid $\mathcal Q$ determines the reducible cubics through \emph{any} of its representations, and hence the range of the summations in \eqref{eq:rpversusepsilonp}. In this section, we determine the values $\textrm{mult}_{\mathcal D}(C)$ of the summands.

The multiplicity of a reducible cubic on the discriminant $\mathcal D$ is invariant under the action of $\PGL_2$ on $\mathbb{P}S^3 \mathbb{C}^3$. Hence, the contribution of a reducible cubic $C$ to the number $m_\p$ only depends on the orbit of $C$ under this action. The characterization of these orbits is classical; we refer to  \cite{KoganMaza:ComputationCanonicalFormsTernaryCubics} for a modern reference.
For each orbit, we record in \autoref{tab:orbits} its name, a representative $C \in \mathbb{Q}[x_0,x_1,x_2]$, and the multiplicity $\mult_{\calD}(C)$. We refer to \autoref{sec:tighterboundtangentcones} for a proof of these multiplicities (see \autoref{rem:Multiplicityproof}). Each orbit as well as the orbit-closure containments are illustrated in  \autoref{fig: containment singular plane cubics}. 

\begin{table}[!htpb]
\begin{tabular}{|r|c|c|}\hline
Name & Representative & Multiplicity \\ \hline \hline 
Nodal Cubic &$x_0x_1^2-x_2^2(x_2-x_0)$ & $1$ \\
Cuspidal Cubic & $x_0x_1^2-x_2^3$ & $2$\\
Conic+Secant Line & $x_0(x_0^2+x_1^2-x_2^2)$& $2$\\
Conic+Tangent Line &$x_0(x_0x_1+x_2^2)$ & $3$\\
Triangle & $x_0x_1x_2$ & $3$\\
Asterisk &$x_0x_1(x_0+x_1)$ & $4$\\
Double Line + Transverse Line & $x_0^2x_1$ &$6$ \\
Triple Line & $x_0^3$& $8$ \\ \hline
\end{tabular}
\caption{A name, representative, and the multiplicity for each singular cubic orbit.}
\label{tab:orbits}
\end{table} 
\begin{figure}[!htpb]
 \begin{tikzpicture}[every text node part/.style={align=center}, scale=0.9]
  \draw (-2,0)--(-6,-1.5)--(-1,-3)--(-2,-6)--(-2,-8.5);
  \draw (-2,0)--(2,-1.5)--(5,-3)--(2,-4)--(-2,-6);
  \draw (2,-1.5)--(-2,-3);
  \node[draw, fill= white] at (-2,0) () {\includegraphics[scale=0.2]{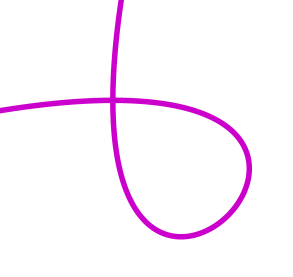} \\ \text{nodal cubic}};
  \node[draw, fill= white] at (-6,-1.5) () {\includegraphics[scale=0.2]{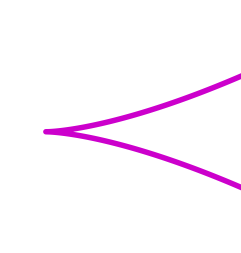} \\\text{cuspidal cubic}};
  \node[draw, fill= white] at (2,-1.5) () {\includegraphics[scale=0.2]{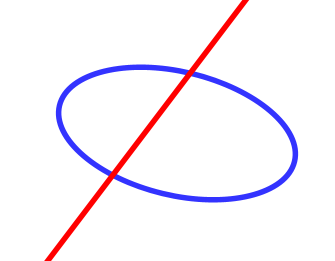}  \\ \text{conic+secant}};
  \node[draw, fill= white] at (-2,-3) () {\includegraphics[scale=0.23]{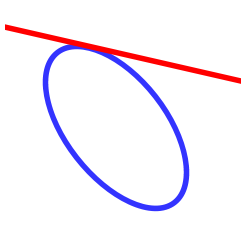}  \\\text{conic+tangent}};
  \node[draw, fill= white] at (5,-3) ()  {\includegraphics[scale=0.2]{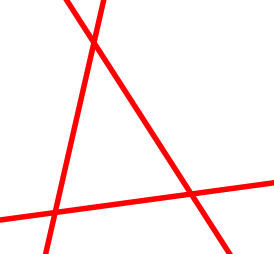} \\ \text{triangle}};
  \node[draw, fill= white] at (2,-4) () {\includegraphics[scale=0.2]{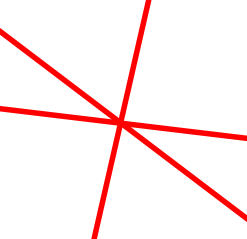} \\ \text{asterisk}};
  \node[draw, fill= white] at (-2,-6) () {\includegraphics[scale=0.2]{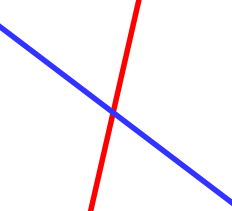} \\ \text{double line+transverse}};
  \node[draw, fill= white] at (-2,-8.5) () {\includegraphics[scale=0.2]{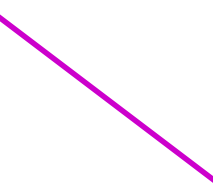} \\ \text{triple line}};
 \end{tikzpicture}
 \caption{Orbit closure containments in the discriminant of singular cubics}\label{fig: containment singular plane cubics}
\end{figure}
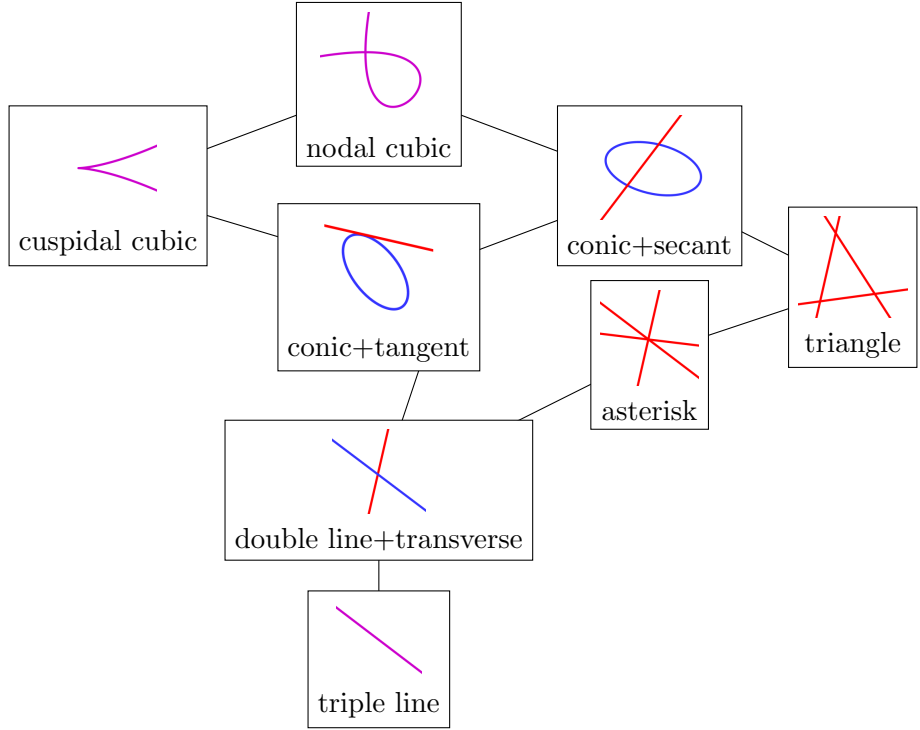

A consequence of \autoref{thm:lowerbounds} along with the orbit-closure containments of singular cubics, imply that each reducible cubic through a generic configuration of a B\'ezoutian quatroid stratum is either a triangle or a conic+secant. In light of this fact, we associate to each quatroid stratum $\mathcal S_{Q}$ a string of symbols of the form $\mydef{\varnothing^i\triangle^j}$ which indicates that a generic representative of $\mathcal S_{\mathcal Q}$ is contained in $i$ reducible cubics of the form conic+secant, and $j$ triangles. These are listed, in black font, in the column \emph{Reducibles} of \autoref{tab:bigtable1} and \autoref{tab:bigtable2}. 

\begin{theorem}
\label{thm:epsilonQcomputation}
Let $\calQ = (\calI,\calJ)$ be a B\'ezoutian quatroid. Then
\begin{align*}
m_\mathcal Q=&2 \cdot | \mathcal I| + 2\cdot |\mathcal J| -2\cdot \Bigl|\Bigl\{\{I,J\} \in \mathcal I \times \mathcal J \mid |I \cap J|=1\Bigr\}\Bigr| \\ &- \Bigl|\Bigl\{\{I_1,I_2\} \subset \mathcal I \mid I_1 \cap I_2 = \emptyset\Bigr\}\Bigr|-\Bigl|\Bigl\{\{I_1,I_2,I_3\} \mid I_1 \cup I_2 \cup I_3 = \{ 1 \vvirg 8\}\Bigr\}\Bigr|. 
\end{align*}
\end{theorem}
\begin{proof}
The proof is an overcount along with a justification of correction terms. For each triple and each sextuple, we count two under the erroneous assumption that every triple and sextuple gives rise to its own conic+secant pair, which by \autoref{tab:orbits} has multiplicity two. This accounts for the first two summands. We have directly overcounted conic+secant pairs by the third summand. 

The fourth and fifth summands cover the cases when the triple did \emph{not} give rise to a conic+secant, but rather, a triangle. In one case, two of the three lines of that triangle were counted by $2|\mathcal I|$, contributing a total of $4$ to $m_{\mathcal Q}$, which should have been a $3$ by \autoref{tab:orbits}. Hence the fourth summand corrects for this. The other case is when all three lines of the triangle were counted by $2|\mathcal I|$, in which case the fourth summand reduced the contribution from $6$ to $4$, and the fifth and final summand corrects this contribution to $3$;  this correction term only counts once.
\end{proof}

\begin{theorem}
\label{thm:upperboundsnottight}
For all $\mathcal Q \in \mathfrak B$ the numbers computed by  \autoref{thm:lowerbounds} agree with the numbers $12-m_\mathcal Q$ except for  the following $24$ quatroid orbits
$$\{\mathcal Q_{10},\mathcal Q_{12},\mathcal Q_{18}
,\mathcal Q_{21},
\mathcal Q_{26},\mathcal Q_{29},\mathcal Q_{31},
\mathcal Q_{38},
\mathcal Q_{47},
\mathcal Q_{56},\mathcal Q_{58}
,\mathcal Q_{72}\}=\{\mathcal Q \in \mathfrak B \mid \mathcal Q \leq \mathcal Q_{10}, \mathcal Q \not\leq \mathcal Q_{77}\}$$
$$\{\mathcal Q_{33},
\mathcal Q_{34},
\mathcal Q_{35},
\mathcal Q_{42},
\mathcal Q_{48},
\mathcal Q_{49},
\mathcal Q_{67},
\mathcal Q_{68},
\mathcal Q_{69},
\mathcal Q_{72},
\mathcal Q_{77}\} = \{\mathcal Q \in \mathfrak B \mid \mathcal Q \leq \mathcal Q_{77}\}.$$
In the above cases, the number $12-m_\mathcal Q$ is precisely \emph{one} more than the associated number in  \autoref{thm:lowerbounds}.
\end{theorem}
\begin{proof}
The number $m_{\mathcal Q}$ is easily computed via the formula in  \autoref{thm:epsilonQcomputation}.
\end{proof}

That the bound from  \autoref{thm:epsilonQcomputation} only fails to be an equality by at most one is a stroke of luck. Consequently, it is enough to show that for each of these $24$ quatroids, a configuration $\p$ represents it only if $\textrm{imult}_C(L_{\p},\mathcal D) > \textrm{mult}_\mathcal D(C)$ for some reducible $C$ through $\p$.

\subsection{A tighter bound for $d_{\mathcal Q}$ through tangent cones of reducible cubics}\label{sec:tighterboundtangentcones}
The expected count of rational cubics through generic $\p \in \mathcal S_{\mathcal Q}$ fails to be correct when the line $L_{\p}$ intersects $\mathcal D$ at some reducible $C \in \mathcal D$ with intersection multiplicity greater than $\textrm{mult}_{\mathcal D}(C)$. Such lines are characterized by the \emph{tangent cone} to $\mathcal D$ at $C$.

\begin{definition}\label{def:tangent space and cone} The \mydef{tangent space} to $\mathcal D$ at $C \in \mathcal D$ is the union of lines $L \subset \mathbb{P}S^3\mathbb{C}^3$ for which 
\[
\textrm{imult}_{C}(L,\mathcal D) >1, 
\]
and is denoted $\mydef{\T_{C}\mathcal D}$. The \mydef{tangent cone} to $\mathcal D$ at $C $ is the union of lines $L \subset \mathbb{P}S^3\mathbb{C}^3$ for which 
\[
\textrm{imult}_{C}(L,\mathcal D) > \textrm{mult}_{\mathcal D}(C), 
\]
and is denoted $\mydef{\TC_{C}\mathcal D}$. Such lines are said to be \mydef{tangent} to $\mathcal D$ at $C$.
\end{definition}

\begin{remark}
We refer to \cite[Ch.\ 2]{Shafarevich} for more standard definitions of tangent spaces and tangent cones. The equivalent \autoref{def:tangent space and cone} streamlines the analysis necessary for our results.
\end{remark}

For $p \in \mathbb{P}^2_{\mathbb{C}}$ we let $\mydef{H_p} \subset \mathbb{P}S^3 \mathbb{C}^3$ be the hyperplane of cubics which vanish at $p$. The description of tangent spaces at smooth points of the discriminant is classical.

\begin{remark}\label{rem:TangentSpaceSmoothPoint}
A point $C \in \mathcal D$ is a smooth point of $\mathcal D$ if and only if it is a nodal cubic. If $C$ is smooth on $\mathcal D$ and $p$ is the node of $C$ then 
\[
\textrm{T}_C\mathcal D = \TC_C\mathcal D = H_p.  
\]
For details, see \cite[Example 1.2.3]{Dolg}.
\end{remark}

 \autoref{tab:orbittangentcones} extends the tangent cone description of  \autoref{rem:TangentSpaceSmoothPoint}  beyond nodal cubics. Write $\mydef{kH_{p}}$ for the hyperplane $H_p$ counted with multiplicity $k$. Moreover, if $\ell \in \mathbb{P}S^1{\mathbb{C}}^3$ is a linear form, write $\mydef{\mathcal D_{\ell}'}$ for the subset of $\mathbb{P}S^3\mathbb{C}^3$ of cubics whose intersection with the line $\{ \ell = 0\}$ is singular.

\begin{proposition}
\label{prop:tangentcones}
The information in \autoref{tab:orbittangentcones} is correct.
\end{proposition}
\begin{proof}
The  \texttt{Macaulay2} script \texttt{tangentConesPlaneCubics.m2} proves the result symbolically. 
\end{proof}

\begin{table}[!htpb]
\begin{tabular}{|c|c|c|c|}
\hline 
Cubic Orbit & Curves & Points & Tangent Cone\\ \hline \hline 
Nodal Cubic & $C$ & $p$: the node of $C$ & $H_p$ \\
Cuspidal Cubic &$C$ &  $p$: the cusp of $C$ & $2H_p$ \\
Conic+Secant Line & $q$: conic, $\ell$: line & $(p_1,p_2)=q\cap \ell$ & $H_{p_1} \cup H_{p_2}$\\
Conic+Tangent Line & $q$: conic, $\ell$: line& $p=q\cap \ell$ & $3H_p$ \\
Triangle & $\ell_1,\ell_2,\ell_3$: lines & $p_k=\ell_i\cap \ell_j$ & $H_{p_1} \cup H_{p_2} \cup H_{p_3}$ \\
Asterisk &$\ell_1,\ell_2,\ell_3$: lines& $p=\ell_1\cap \ell_2\cap \ell_3$ & $4H_p$ \\
Double Line+Transverse Line & $\ell_1$: double line, $\ell_2$: line & $p=\ell_1\cap \ell_2$ & $2H_p \cup \mathcal D'_{\ell}$ \\
Triple Line & $\ell$: line & & $2\mathcal D'_{\ell}$ \\ \hline
\end{tabular}
\caption{Descriptions of tangent cones to the discriminant $\mathcal D$.}
\label{tab:orbittangentcones}
\end{table}

\begin{remark}
\label{rem:Multiplicityproof}
By definition, the multiplicity of a point $C$ on the discriminant $\mathcal D$ is the same as the multiplicity of $C$ on $\TC_C \mathcal D$. Since $\TC_C \mathcal D$ is a cone over $C$, the multiplicity of $C$ in $\TC_C$ coincides with the degree of $\TC_C$. Hence one may take the \emph{Tangent Cone} column of  \autoref{tab:orbittangentcones} as the proof of the \emph{Multiplicity} column of  \autoref{tab:orbits}. We remark that $\mathcal D'_{\ell}$ has degree four since it is a cone over the variety of singular binary cubics.
\end{remark}

\begin{corollary}
\label{cor:TangentConeSingularCubic}
Let $\p \in \mathcal P$. The pencil $L_{\p}$ is tangent to $\mathcal D$ at $C \in L_{\p}$ if and only if $C$ is singular at one of the points in $\p$.
\end{corollary}
\begin{proof}
This follows from the characterization of the tangent cones given in \autoref{tab:orbittangentcones}. 
\end{proof}

\begin{proposition}
If $\p$ represents $\mathcal Q$ for $\mathcal Q \leq \mathcal Q'$ where $\mathcal Q'$ is in the same orbit as $\mathcal Q_{10}$, then $L_{\p}$ tangent to the discriminant. 
Conversely, if $\p$ is a generic point on $\mathcal S_{\mathcal Q}$ for some $\mathcal Q \not\leq \mathcal Q'$ for every $\mathcal Q'$ in the orbit of $\mathcal Q_{10}$, then $L_{\p}$ is not tangent to the discriminant.
\end{proposition}
\begin{proof} Without loss of generality, assume $\calQ'=\calQ_{10}$.
Recall that $\mathcal Q_{10} = (\{123\},\{145678\})$ and observe that $\ell_{123}\cdot q_{145678}$ is a reducible cubic $C$ through any realization. It is singular at $p_1$ and so $L_{\p}$ is tangent to $\mathcal D$ at $C$. 
Conversely, for all other quatroids, note that the $d_{\p}$ computed from  \autoref{thm:lowerbounds} agrees with $12-m_\mathcal Q$ as computed in \autoref{thm:upperboundsnottight}. Therefore,  $m_\p = r_\p$  and the result follows.
\end{proof}

\begin{proposition}
\label{prop:equivalentNonreduced} 
Let $\p \in \mathcal P$ represent a B\'ezoutian quatroid. The following are equivalent:
\begin{enumerate}
\item $L_{\p}$ is tangent to $\mathcal D$,
\item There exists a cubic $C \in L_{\p}$ which is singular at a point of $\p$,
\item The base locus $Z(L_{\p})$ is nonreduced, with one component of length exactly two,
\item $d_{\p} \leq 12-m_{\p} -1 $.
\end{enumerate}
\end{proposition}
\begin{proof}
The equivalence of parts (1) and (2) is \autoref{cor:TangentConeSingularCubic}. The equivalence of parts (1) and (4) follows directly from the definitions of multiplicity and tangent cone. 

If part (2) is true, then so is part (3) since a singular point on a cubic has multiplicity at least two, and thus intersects any other cubic of $L_\bfp$ in a scheme of length at least $2$. That scheme has length exactly two, because eight points of the base locus are necessarily distinct.  To see that part (3) implies part (2), note that if the base locus is not reduced, then all cubics in $L_{\p}$ have the same tangent at some point $p$ in $\p$. Up to the action of $\PGL_2$, assume $p=[1:0:0]$ and write the equations of two cubics $C_0,C_1 \in L_{\p}$ as $x_1x_0^2+F_0$ and $x_1x_0^2+F_1$ where $x_1=0$ is the tangent line of $C_0$ and $C_1$ at $p$ and $F_i$ involves subquadratic terms in $x_0$. The cubic $C_0-C_1$ is singular at $p$. 
\end{proof}

\begin{corollary}
\label{cor:1077Nonreduced}
Let $\mathcal Q$ be a B\'ezoutian quatroid. 
The base locus $Z(L_{\p})$ is nonreduced for every $\p \in \mathcal S_{\mathcal Q}$ if and only if  
$\mathcal Q \leq \mathcal Q'$ for some $\mathcal Q'$ in the orbit of $\mathcal Q_{10}$.
\end{corollary}

We conclude this section by summarizing the results we proved before and completing the proof of the main results of the paper. 

\begin{theorem}
\label{thm:MainTheorem}
For every $\mathcal Q \in \mathfrak Q$, the numbers in the $d_\mathcal Q$ columns of  \autoref{tab:bigtable1} and \autoref{tab:bigtable2} correctly give the number of rational cubics through a generic point on $\mathcal S_{\mathcal Q}$.
\end{theorem}
\begin{proof}
As this is the main result of the paper, we summarize the steps we took to achieve it.

For each representable quatroid $\mathcal Q$, we found an explicit representative $\p \in \mathcal S_{\mathcal Q}$ and applied \autoref{algo:symbolicrational} to $\p$ to compute the number $d_{\p}$ of rational cubics through that representative. This gave the lower bound $ d_\bfp \leq d_{\mathcal Q}$, as shown in \autoref{thm:lowerbounds}. Along the way, we showed that the reducible cubics through those representatives are either triangles or conic+secant lines, showing that such reducible orbits are generic. This allowed us to obtain the formula of \autoref{thm:epsilonQcomputation} for $m_{\mathcal Q}$. Together, these results imply 
\[
d_\p \leq d_{\mathcal Q} \leq 12-m_\mathcal Q.
\]
\autoref{thm:upperboundsnottight} characterizes when this is an equality. When it is not, \autoref{cor:1077Nonreduced} implies the generic base locus is not reduced and so by \autoref{prop:equivalentNonreduced} the inequality tightens to
\[
d_{\p} \leq d_{\mathcal Q} \leq 12-m_{\mathcal Q}-1
\]
which again by \autoref{thm:upperboundsnottight} is an equality. 
\end{proof}
In \autoref{tab:bigtable1} and \autoref{tab:bigtable2}, those B\'ezoutian quatroids satisfying any of the conditions in \autoref{prop:equivalentNonreduced} are indicated by one additional symbol in their \emph{Reducibles} column. This symbol, written in blue, indicates the type of the reducible cubic at which the intersection multiplicity of the discriminant and $L_{\p}$ is (one) larger than its multiplicity, for a generic quatroid representative $\p$. Hence, given a string of symbols of the form $\varnothing^i \triangle^j \mydef{\square}$ for $\square \in \{\varnothing,\triangle\}$, one may calculate the value in the column $d_{\mathcal Q}$ by $d_\mathcal Q = 12 - 2i -3j -\textbf{1}_{\square}$, where $\textbf{1}_{\square}$ is the indicator function of the presence of a third symbol.

\subsection{The poset of B\'ezoutian quatroids}

We conclude our work with some observations on the poset of B\'ezoutian quatroids given by the order $\leq$. In order to display these results, instead of working with the $544748$ B\'ezoutian quatroids, we work modulo the $\mathfrak S_8$-symmetry, and use the $76$ orbits instead. This allows us to fully illustrate the induced poset   in \autoref{fig: poset}, where $\mathcal Q \leq \mathcal Q'$ if there exists $\mathcal Q''$ in the orbit of $\mathcal Q'$ such that $\mathcal Q \leq \mathcal Q''$. 

We prove that the dimension of the realization space defines a grading on this poset of B\'ezoutian quatroids, in the sense of \cite[Sec. 3.1]{Stanley:EnumerativeCombinatoricsVol1}. Recall that $\mathcal Q'$ \mydef{covers} $\mathcal Q$ when $\mathcal Q'$ is minimally larger than $\mathcal Q$. We have the following result.

\begin{theorem}\label{thm: poset theorem}
 The set $\frakB$ with the order relation $\leq$ on quatroids is represented in \autoref{fig: poset}. It is partitioned  into nine layers $\frakB_0 \vvirg \frakB_8$: $\calQ \in \frakB_j$ if and only if $\dim \calS_\calQ = 8+j$. Moreover:
 \begin{enumerate}[(i)]
\item $\calQ_1$ is the only quatroid in $\frakB_8$,
\item $\calQ_{41}$ is the only quatroid in $\frakB_0$,
\item $\frakB$ is graded by dimension: if $\calQ'$ covers $\calQ$ then $\dim \calS_{\calQ'} = \dim \calS_{\calQ} +1$,
\item $\frakB$ is graded by \emph{number of conditions}: if $\calQ'$ covers $\calQ$ then $| \calI '| + |\calJ' | = |\calI| + |\calJ|+1$.
 \end{enumerate}
\end{theorem}

The proof of \autoref{thm: poset theorem} relies on the following technical result, related to \autoref{rem:OrderIsCombinatorial}: on the subset of B\'ezoutian quatroids, the second condition listed in \autoref{rem:OrderIsCombinatorial} completely characterizes the order relation.
\begin{lemma}\label{lemma:inclusion bezoutian}
 Let $\calQ = (\calI,\calJ), \calQ' = (\calI',\calJ')$ be B\'ezoutian quatroids. The following are equivalent:
 \begin{enumerate}[(i)]
  \item $\calQ \leq \calQ'$,
  \item $\calS_{\calQ} \subseteq \bar{\calS_{\calQ'}}$,
  \item $\calI' \subseteq \calI$ and $\calJ' \subseteq \calJ \cup \calI^2$, where $\calI^2$ is the set of all subsets of $\{1 \vvirg 8\}$ arising as the union of two elements of $\calI$.
 \end{enumerate}
\end{lemma}
\begin{proof}
 The equivalence of (i) and (ii) is the definition of the order relation. The fact that (iii) implies (ii) follows from \autoref{rem:OrderIsCombinatorial}. It remains to show that (ii) implies (iii). If $\calS_{\calQ} \subseteq \overline{\calS_{\calQ'}}$, every $\bfp \in \calS_{\calQ}$ satisfies the linear and quadratic relations imposed by $\calQ'$. This immediately implies $\calI' \subseteq \calI$. To conclude, let $J \in \calJ'$, and without loss of generality assume $J = 123456$. Since $\calQ$ is B\'ezoutian, there is a unique quadric $q_{123456}$ through a generic $\bfp \in \calS_{\calQ}$. If this quadric is a conic, then $ 123456 \in \calJ$. Otherwise it is the union of two lines: since $\calQ$ is B\'ezoutian, no four points lie on a line. Hence, up to relabeling, $q_{123456} = \ell_{123}\ell_{456}$ showing $123,456 \in \calI$ and $123456 \in \calI^2$ as desired.
\end{proof}

A technical consequence of \autoref{lemma:inclusion bezoutian} is that the \emph{closure} step in \texttt{step 6} of  \autoref{algo:allConicalExtensions} does not alter the pair, whenever the quatroid that is being generated is B\'ezoutian. Implicitly, this fact yields the grading of $\frakB$ by number of conditions, as one can see in the proof of \autoref{thm: poset theorem}.

\begin{proof}[Proof of \autoref{thm: poset theorem}]
We first prove statement (iv). Let $\calQ = (\calI,\calJ), \calQ'=(\calI',\calJ')$ be quatroids and suppose $\calQ'$ covers $\calQ$. 
By \autoref{lemma:inclusion bezoutian},  $\calI' \subseteq \calI$ and $\calJ' \subseteq \calJ \cup \calI^2$. We consider several cases:
\begin{itemize}
 \item Suppose there is a condition $I \in \calI \setminus \calI'$. 
 \begin{itemize}
 \item If for every $J' \in \calJ$ we have $I \not\subseteq J'$ then we claim $\calI = \calI' \cup \{ I \}$ and $\calJ = \calJ'$: this is immediate because  $\calQ'' = (\calI' \cup \{ I \}, \calJ')$ satisfies B\'ezout's criteria and $\mathcal Q'$ is \emph{minimally} larger than $\mathcal Q$.
 \item Otherwise, set $J'_* = \{J' \in \mathcal J' \mid I \subseteq J'\}$ and $I'_* = \{J'\setminus I \mid J' \in J'_*\}$. Consider $\mathcal Q'' = (\calI' \cup I'_* \cup \{I\},\mathcal J'\setminus J'_*)$. By \autoref{lemma:inclusion bezoutian} $\mathcal Q'' \leq \mathcal Q'$. Moreover, there is exactly one more condition in $\mathcal Q''$ than there is in $\mathcal Q'$. One may verify that $\mathcal Q''$ is B\'ezoutian and so, as before, we conclude $\mathcal Q''=\mathcal Q$. 
\end{itemize} 
\item If $\calI = \calI'$, then $\calJ' \cap \calI^2 = \calJ' \cap {\calI'}^2 = \emptyset$ because $\calQ'$ satisfies B\'ezout's weak criteria. This implies $\calJ' \subseteq \calJ$, and by minimality $\calJ = \calJ' \cup \{ J \}$ for some $J$.
\end{itemize}
This shows statement (iv). Statement (iii) is then a consequence of statement (iv) and the irreducibility of the strata $\calS_{\calQ}$. The structure of the poset in \autoref{fig: poset} and statement (i) and (ii) follow by direct computation, which can be done in a purely combinatorial way using \autoref{lemma:inclusion bezoutian}.
\end{proof}

\begin{figure}
 \includegraphics[scale=.12]{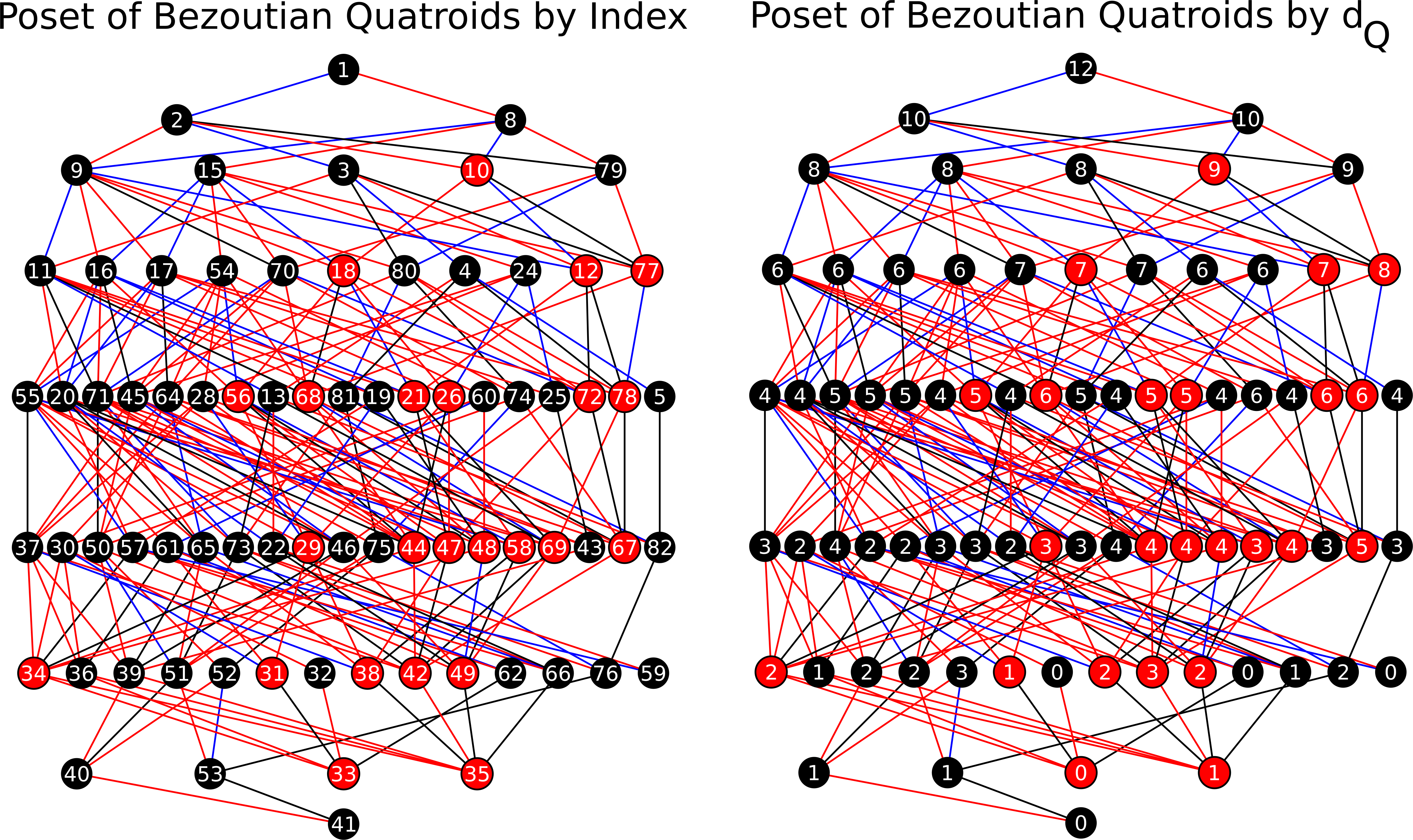}
 \caption{The poset of B\'ezoutian quatroids: Red nodes are quatroids $\calQ \leq \calQ_{10}$. A red connection indicates that the quatroids differ by precisely one element, in $\calI$. A blue connection indicates the difference is one element, in $\calJ$. A black connection indicates any other difference, occurring from sextuples ``breaking'' into lines. The left picture uses the quatroid orbit indices as labels, and the right picture uses the corresponding values of $d_{\calQ}$ as labels. Note that the the value $d_\calQ$ never increases moving down in the poset; moreover, if it does not decrease, then a black node becomes a red node}\label{fig: poset}
\end{figure}

\section{Concluding Remarks}
\label{sec:conclusion}
\subsection{Positive certificates for non-rationality}\label{subsec: nonrat certificate}

 \autoref{thm:MainTheorem} allows one to design \emph{non-rationality certificates} for cubic curves through points in special position. More precisely, there are several quatroid strata $\mathcal S_{\mathcal Q}$ satisfying $d_\calQ = 0$: in these cases, for any $\p \in \mathcal S_{\mathcal Q}$, there are no rational cubics through $\bfp$. The same holds passing to the closure, proving the following result.

\begin{theorem}\label{thm:norationalcubic}
Let $\p \in \mathcal S_{\mathcal Q}$ for some $\calQ \leq \calQ'$ with 
\[
\mathcal Q'  \text{ in the orbit of an element of } \{\mathcal Q_6, \mathcal Q_{32}, \mathcal Q_{41}, \mathcal Q_{59}, \mathcal Q_{62}, \mathcal Q_{121}\}.
\]
Then there are no rational cubics passing through $\bfp$. 
\end{theorem}
 
We point out that if $\calQ \leq \calQ_6$ or $\calQ \leq \calQ_{121}$, then $\calQ$ is non-B\'ezoutian: in this case, there are no irreducible cubics at all passing through $\bfp$.

We remark that $\mathcal Q_{59}$ and $\mathcal Q_{62}$ are exhaustive B\'ezoutian quatroids so the matroid underlying the (reduced) base locus $Z(L_{\p})$ of any $\bfp \in \calS_{\calQ}$ is well-defined; in these cases, this is the non-Fano matroid $\mathcal Q_{32}$. Configurations of $\calQ_{59}$ appear in \cite{Fie:PencilsCubicsEightPoints}, in the study of the topology of singular cubics.
In a sense, configurations $\bfp$ as in \autoref{thm:norationalcubic} are \emph{forbidden configurations} on a rational cubic. More precisely, we have the following consequence.
\begin{corollary}
\label{cor:positivecertificate}
Let $C \subset \mathbb{P}_{\mathbb{C}}^2$ be an irreducible cubic containing a configuration $\p \in \calS_\calQ$ where
\[\calQ \leq \calQ' \quad  \text{ for some }\mathcal Q' \text{ in the orbit of an element of }  \quad \{\mathcal Q_6, \mathcal Q_{32}, \mathcal Q_{41}, \mathcal Q_{59}, \mathcal Q_{62}, \mathcal Q_{121}\}.
\]
Then $C$ is not rational.
These forbidden quatroids are illustrated in \autoref{fig:norationalcubics}.
\end{corollary}

\begin{figure}[!htpb]
\includegraphics[scale=0.4]{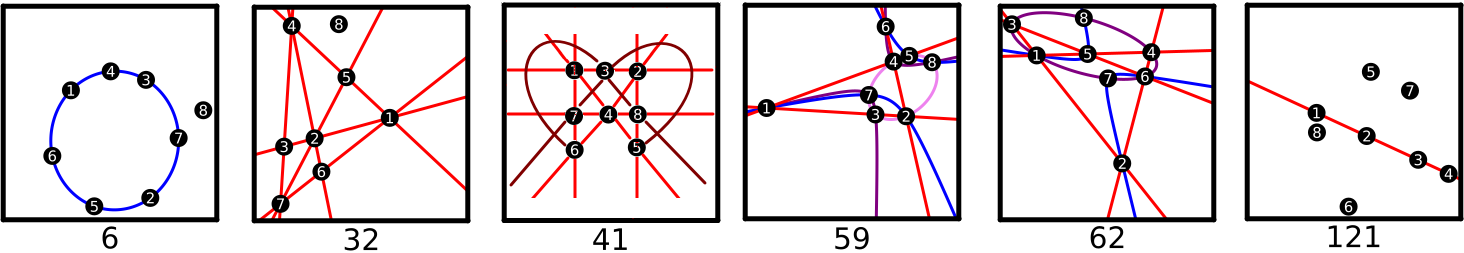}
\caption{The six minimal quatroid orbits forbidding rational cubics through a representation.}
\label{fig:norationalcubics}
\end{figure}

To the extent of our knowledge, the \emph{existence} of forbidden configurations guaranteeing non-rationality gives a novel way to prove non-rationality of a variety. In this sense, these are \emph{positive certificates} of non-rationality.
We leave open the problem of studying positive non-rationality certificates for curves of higher degree, and varieties of higher dimension.

\subsection{A finer stratification}

We remark that the values $d_\calQ$ are the number of rational cubics through a \emph{generic} configuration $\bfp \in \calS_\calQ$. A natural question is whether this is the number of rational cubics through \emph{every} configuration $\bfp \in \calS_{\calQ}$. This is not the case, as observed in the following construction. 

Consider the configuration $\bfp_z$, depending on a complex parameter $z$ and described by the matrix
 \begin{equation*}
\bfp_z = \left[ \begin{array}{cccccccc}
 1&4&9&16&25&36&1&1\\
 1&2&3&4&5&6&1&-1\\
 1&1&1&1&1&1&0&z
\end{array} \right].
 \end{equation*}
For generic $z \in \bbC$, the configuration $\bfp_z$ represents quatroid $\calQ_{2}$: the points $p_1 \vvirg p_6$ lie on the conic $x_1^2 - x_0x_2 = 0$, and there is no other linear or quadratic relation. For generic $z$, $L_{\bfp_z}$ intersects $\calD$ in the reducible cubic $C = q_{123456}\ell_{78}$ with multiplicity $2 = \mult_{\calD}(C)$, and $10$ additional rational cubics. Hence $d_{\bfp_z} = d_\calQ = 10$.

Let $\bfp^*$ be the configuration for $z = 0$, which represents $\calQ_2$ as well. In this case, the line $\ell_{78}$ is tangent to the conic $q_{123456}$. The pencil $L_{\bfp^*}$ intersects $\calD$ at $C$ with intersection multiplicity $3 = \mult_\calD(C)$ and at an additional $9$ distinct rational cubics.

This construction shows that there are special configurations $\bfp^*$ on the quatroid stratum $\calS_{\calQ_2}$ satisfying $d_{\bfp^*} < d_{\calQ_2}$. A similar example can be constructed on the other maximal non-uniform quatroid $\calQ_{8}$. A slightly different example can be constructed on $\calQ_{78}$: in this case, for a generic choice of $\bfp \in \calS_{\calQ_{78}}$, the only reducible cubic on $L_{\bfp}$ is a triangle $C = \ell_1\ell_2\ell_3$ and $L_{\bfp}$ intersects $\calD$ in $C$ with multiplicity $3 = \mult_{\calD}(C)$ and in $9$ rational cubics, so that $d_\bfp = d_{\calQ_{78}} = 9$. There is a locus in $\calS_{\calQ_{78}}$ where the triangle $C$ degenerates to an asterisk: here $\mult_\calD(C) = 4$ and $d_\bfp = 8 < d_\calQ$.

This phenomenon occurs on other strata. Interestingly, there are strata, such as $\calS_{\calQ_{40}}$, with the property that $d_\calQ = 1$ and with a locus of configurations $\bfp$ such that $d_{\bfp} = 0$. Understanding these loci, with no rational cubics through them, would provide other \emph{non-rationality certificates}, in the sense of \autoref{subsec: nonrat certificate}.

We expect the special loci described in this section always arise with either a conic$+$line pair $C = q\ell$ degenerating to a conic$+$tangent pair, or with a triangle $C = \ell_1\ell_2\ell_3$ degenerating to an asterisk. The second phenomenon is linear but it cannot be detected simply by the matroid underlying $\bfp$; however, it is detected by the underlying \emph{discriminantal arrangement} \cite{ABFKSSL:LikelyhoodDegenerations}. We plan to further investigate higher order versions of the discriminantal arrangement in future work. 

\subsection{Toward higher degree and higher genus}
Given two integers $d$ and $g$, one may consider the locus $V_{d,g} \subseteq \bbP S^d \bbC^3$ of plane curves of degree $d$ and genus $g$. In \cite{Harris:SeveriProblem}, Harris answered a conjecture of Severi \cite{Severi:Vorlesungen}, proving that $V_{d,g}$ is irreducible of dimension $3d+g-1$. This inspired a body of work surrounding the natural enumerative problem: 
\begin{center}
\emph{Given $3d+g-1$ points in $\bbP_{\mathbb{C}}^2$, how many elements of $V_{d,g}$ interpolate them?}
\end{center}
For generic points, this amounts to computing the degree of (the closure of) $V_{d,g}$: when $g = 0$, the answer is Kontsevich's formula \cite{Kontsevich1994} and a recursive formula was given for any genus in \cite{CapHar:CountingPlaneCurves}. The present paper dealt with non-generic instances of the problem when $(d,g)=(3,0)$, that is the case of rational cubics.  Throughout our investigation, we enjoyed a number of nice properties:
\begin{itemize}
\item the cubic discriminant is the union of rational cubics and reducible cubics,
\item the discriminant is cut out by a manageable polynomial,
\item the number of rational cubics through eight generic points is of modest size,
\item there is a known classification of orbit closures of rational and reducible cubics.
\end{itemize}
We propose the study of curves of higher degree and genus, with similar, higher order methods.

For instance, there are $620$ rational quartics through $11 = 3\cdot 4 + 0 - 1$ generic points. Despite all the ``good'' properties mentioned above failing in this setting, numerical methods can generate experimental data. We now consider the problem of computing all rational quartics through two interesting matroid strata on $11$ points.
A \mydef{L\"uroth quartic} is any quartic curve which goes through the $10={{5}\choose{2}}$ intersection points of five lines. The set of L\"uroth quartics forms an unwieldy hypersurface of degree $54$ in the space $\mathbb{P}S^4\mathbb{C}^3$ of homogeneous quartics \cite{NumericalNP,OttSer:HypersurfaceLuroth,Ott:Computational}. These quartics inspire the definition of the \mydef{L\"uroth matroid} on $11$ points, which we write in our quatroid format:
\[
\mydef{\mathcal Q_{\textrm{L\"uroth}}} = (\{1367,159\underline{10},246\underline{10},2578,3489\},\{\}).
\]
Numerical computations suggest that there are $40$ rational quartics through a L\"uroth configuration.  In \autoref{fig:Luroth} we display one example of a L\"uroth configuration and the $12$ real rational quartics which interpolate it.
\begin{figure}[!htpb]
\includegraphics[scale=0.37]{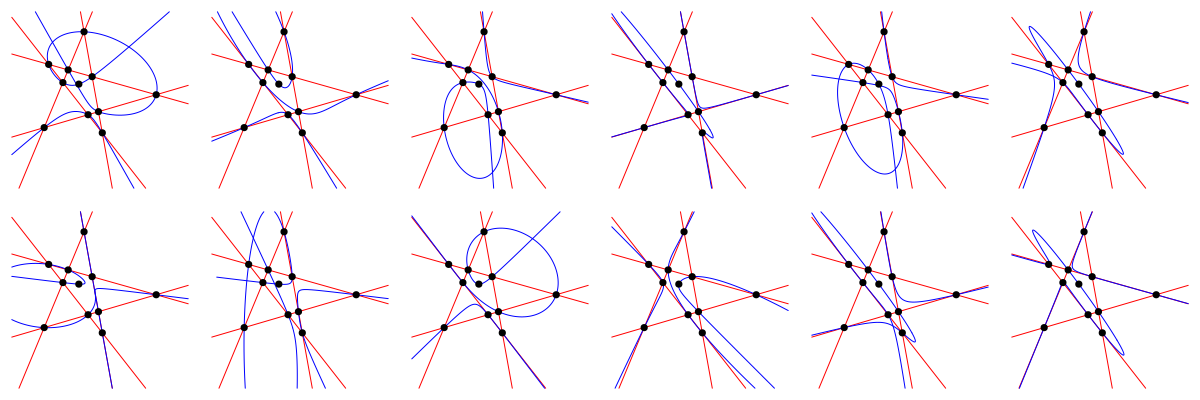}
\caption{Twelve real rational quartics (of the forty over $\mathbb{C}$) through a  L\"uroth configuration of $11$ points. The five lines are drawn in red and the rational quartics in blue.}\label{fig:Luroth}
\end{figure}

A more restricted matroid is suggested by a \mydef{regular pentagon configuration} \cite{Ziegler} which imposes five additional line conditions on the eleventh point. The underlying matroid is 
\[
\mydef{\mathcal Q_{\textrm{Pentagon}}} = (\{1367,159\underline{10},246\underline{10},2578,3489,23\underline{11},18\underline{11},69\underline{11},7\underline{10}\, \underline{11},45\underline{11}\},\{\})
\]
Such a configuration can  be realized over $\mathbb{R}$ but not $\mathbb{Q}$, and it is unique up to the action of $\PGL_2$. We invite the reader to draw this configuration. A numerical calculation suggests that there are $10$ rational quartics interpolating the regular pentagon configuration, none of which are real.

\newpage 
\section*{Appendix}\label{appendix}

{\footnotesize{
\begin{table}[!bhtp]
\begin{tabular}{|p{0.36\linewidth}|p{0.49\linewidth}|p{0.15\linewidth}|}\hline
Name  & Description & Relevant  Results\\ \hline  

\begin{tabular}{c}\\ {\texttt{Quatroids.jl}}\end{tabular}   & The \texttt{julia} package \texttt{Quatroids.jl}. All commands loaded by this package are indicated by \mydef{blue font} & \\ \hline

\begin{tabular}{c}\\ \alert{\texttt{SimpleMatroids38.txt}} \end{tabular}  & File extracted from the database \url{https://www-imai.is.s.u-tokyo.ac.jp/~ymatsu/matroid/} listing all simple rank $3$ matroids on eight elements &\begin{tabular}{c}  \end{tabular}\\ \hline

\begin{tabular}{c}\\ \mydef{\texttt{GenerateAllMatroids()}}  \end{tabular} & Parses the file \texttt{SimpleMatroids38.txt} to obtain an exhaustive list of simple $\mathbb{C}$-representable matroids of rank at most three on eight elements &\begin{tabular}{c} \\ \autoref{thm:candidatequatroids} \end{tabular}\\ \hline 

\begin{tabular}{c}\\ \mydef{\texttt{AllConicExtensions(M)}} \end{tabular} & Runs \autoref{algo:allConicalExtensions} on a matroid $M$ represented by nonbases of size three &\begin{tabular}{c} \autoref{thm:candidatequatroids} \end{tabular}\\ \hline 
 
\begin{tabular}{c}\\ \mydef{\texttt{GenerateAllCandidateQuatroids()}} \end{tabular}  & Runs \autoref{algo:allConicalExtensions} all $\mathbb{C}$-representable matroids of rank at most three on eight elements & \begin{tabular}{c}\autoref{thm:candidatequatroids}\\ \autoref{tab:bigtable1} \\ \autoref{tab:bigtable2}  \end{tabular}\\ \hline 
 
\begin{tabular}{c}\\ \mydef{\texttt{OrbitSizes()}}  \end{tabular}& Computes the sizes of each orbit of candidate quatroids & \begin{tabular}{c}\autoref{thm:candidatequatroids} \\  \autoref{cor:orbitcorollary} \\ \autoref{tab:bigtable1} \\ \autoref{tab:bigtable2} \end{tabular}\\ \hline 

\begin{tabular}{c}\\ \mydef{\texttt{Bezoutian()}}  \end{tabular}& Returns a boolean vector of length $126$ whose $i$-th entry indicates whether $\mathcal Q_i$ is B\'ezoutian & \begin{tabular}{c}\autoref{def:Bezoutian}\\ \autoref{tab:bigtable1} \\ \autoref{tab:bigtable2} \end{tabular} \\ \hline 
 
\begin{tabular}{c}\\ \alert{\texttt{RationalRepresentatives.txt}}  \end{tabular} & A file, given in \texttt{julia} (\texttt{.txt}) and \texttt{Macaulay2} (\texttt{.m2}) format, whose $i$-th line is a $3 \times 8$ matrix of integers whose columns represent $\mathcal Q_i$ &\begin{tabular}{c}\\ \autoref{thm:representability}\end{tabular} \\ \hline 
 
\begin{tabular}{c}\\ {\texttt{TestingRepresentatives.m2}}  \end{tabular}& A \texttt{Macaulay2} script which confirms that each representative in \texttt{RationalRepresentatives.txt} represents the quatroid claimed & \begin{tabular}{c}\\ \autoref{thm:representability}\end{tabular}  \\ \hline 
 
\begin{tabular}{c}\\ \mydef{\texttt{ReducedBaseLocus(Q)}} \end{tabular}& Indicates if the quatroid $\mathcal Q$ has reduced base locus due to \autoref{cor:ExhaustiveReduced} & \begin{tabular}{c} \autoref{lem:Q63} \end{tabular}\\ \hline 
 
\begin{tabular}{c}\\ \mydef{\texttt{QuatroidReductions()}} \end{tabular}& Iteratively reduces each quatroid based on the four conditions described in \autoref{sec:IrreducibilityOfStrata} &\begin{tabular}{c}  \autoref{lem:ABCreductions} \end{tabular}\\ \hline 
 
\begin{tabular}{c}\\ \mydef{\texttt{Modifications(Q)}} \end{tabular}& Computes all $\mathcal Q'$ such that $\mathcal Q \leadsto \mathcal Q'$ &  \begin{tabular}{c} \autoref{lem:Modifications} \end{tabular}  \\ \hline

\begin{tabular}{c}\\ {\texttt{IrreducibilityQ21.m2} etc} \end{tabular}& A \texttt{Macaulay2} script which establishes the irreducibility of $\mathcal Q_{21}$ as described in the proof of \autoref{thm: explicit cases irreducible}. Similar files exist for $\mathcal Q_{25},\mathcal Q_{38},\mathcal Q_{58},\mathcal Q_{101},$ and $\mathcal Q_{123}$ &\begin{tabular}{c}\\ \autoref{thm: explicit cases irreducible} \end{tabular}  \\ \hline 

\begin{tabular}{c}\\ \alert{\texttt{cubicInvs.m2}} \end{tabular}& A \texttt{Macaulay2} file containing the cubic discriminant & \begin{tabular}{c} \autoref{algo:symbolicrational} \end{tabular} \\ \hline 
 
\begin{tabular}{c}\\ {\texttt{countReducibleCubics.m2}} \end{tabular}& Applies \autoref{algo:symbolicrational} to each of the rational representatives of B\'ezoutian quatroids in \texttt{RationalRepresentatives.m2} & \begin{tabular}{c} \autoref{thm:lowerbounds} \\   \autoref{tab:bigtable1} \\ \autoref{tab:bigtable2} \end{tabular} \\ \hline 
 
\begin{tabular}{c}\\ \mydef{\texttt{QuatroidsWeakUpperBounds()}}  \end{tabular}& Evaluates the formula for $m_{\mathcal Q}$ from \autoref{thm:epsilonQcomputation} for each B\'ezoutain quatroid $\mathcal Q$ & \begin{tabular}{c}\autoref{thm:upperboundsnottight} \end{tabular} \\ \hline

 \begin{tabular}{c} \mydef{\texttt{ContainedIn10(Q)}} \\ \mydef{\texttt{ContainedIn77(Q)}} \end{tabular}& Checks $\mathcal Q \leq \mathcal Q_{10}$ and $\mathcal Q \leq \mathcal Q_{77}$ respectively & \begin{tabular}{c}\autoref{thm:upperboundsnottight} \end{tabular} \\ \hline 
 
\begin{tabular}{c}\\  {\texttt{tangentConesPlaneCubics.m2}} \end{tabular}& Computes the tangent cones of each singular cubic  & \begin{tabular}{c} \autoref{prop:tangentcones} \\ \autoref{tab:orbits} \\ \autoref{tab:orbittangentcones} \end{tabular} \\
\hline

\end{tabular}
\caption{Descriptions of auxiliary files}
\label{tab:software}
\end{table}}}

{\scriptsize{
\begin{table}[!htpb]
\begin{tabular}{|r|c|rl|c|c|}
 \hline 
 Num  & $ | $ Orb $ | $ &  Lines  &  Conics  &  Reducibles  & $ d_Q $\\ \hline \hline 
$ 1 $ & $ 1 $ & $  {\color{red}{\{\}}} $ & $  {\color{blue}{\{\}}} $ & $  $ & $ 12 $ \\
$ 2 $ & $ 28 $ & $  {\color{red}{\{\}}} $ & $  {\color{blue}{\{123456\}}} $ & $ \varnothing^1 $ & $ 10 $ \\
$ 3 $ & $ 210 $ & $  {\color{red}{\{\}}} $ & $  {\color{blue}{\{123456\,\,123478\}}} $ & $ \varnothing^2 $ & $ 8 $ \\
$ 4 $ & $ 420 $ & $  {\color{red}{\{\}}} $ & $  {\color{blue}{\{123456\,\,123478\,\,125678\}}} $ & $ \varnothing^3 $ & $ 6 $ \\
$ 5 $ & $ 105 $ & $  {\color{red}{\{\}}} $ & $  {\color{blue}{\{123456\,\,123478\,\,125678\,\,345678\}}} $ & $ \varnothing^4 $ & $ 4 $ \\
$ 6 $ & $ 8 $ & $  {\color{red}{\{\}}} $ & $  {\color{blue}{\{1234567\}}} $ & NB & $  {\color{red}{0}} $ \\
$ 7 $ & $ 1 $ & $  {\color{red}{\{\}}} $ & $  {\color{blue}{\{12345678\}}} $ & NB & $  {\color{red}{0}} $ \\
$ 8 $ & $ 56 $ & $  {\color{red}{\{123\}}} $ & $  {\color{blue}{\{\}}} $ & $ \varnothing^1 $ & $ 10 $ \\
$ 9 $ & $ 840 $ & $  {\color{red}{\{123\}}} $ & $  {\color{blue}{\{124567\}}} $ & $ \varnothing^2 $ & $ 8 $ \\
$ 10 $ & $ 168 $ & $  {\color{red}{\{123\}}} $ & $  {\color{blue}{\{145678\}}} $ & $ \varnothing^1 {\color{blue}{\varnothing}} $ & $ 9 $ \\
$ 11 $ & $ 3360 $ & $  {\color{red}{\{123\}}} $ & $  {\color{blue}{\{124567\,\,134568\}}} $ & $ \varnothing^3 $ & $ 6 $ \\
$ 12 $ & $ 840 $ & $  {\color{red}{\{123\}}} $ & $  {\color{blue}{\{124567\,\,345678\}}} $ & $ \varnothing^2 {\color{blue}{\varnothing}} $ & $ 7 $ \\
$ 13 $ & $ 3360 $ & $  {\color{red}{\{123\}}} $ & $  {\color{blue}{\{124567\,\,134568\,\,234578\}}} $ & $ \varnothing^4 $ & $ 4 $ \\
$ 14 $ & $ 168 $ & $  {\color{red}{\{123\}}} $ & $  {\color{blue}{\{1245678\}}} $ & NB & $  {\color{red}{0}} $ \\
$ 15 $ & $ 840 $ & $  {\color{red}{\{123\,\,145\}}} $ & $  {\color{blue}{\{\}}} $ & $ \varnothing^2 $ & $ 8 $ \\
$ 16 $ & $ 3360 $ & $  {\color{red}{\{123\,\,145\}}} $ & $  {\color{blue}{\{124678\}}} $ & $ \varnothing^3 $ & $ 6 $ \\
$ 17 $ & $ 2520 $ & $  {\color{red}{\{123\,\,145\}}} $ & $  {\color{blue}{\{234567\}}} $ & $ \varnothing^3 $ & $ 6 $ \\
$ 18 $ & $ 3360 $ & $  {\color{red}{\{123\,\,145\}}} $ & $  {\color{blue}{\{234678\}}} $ & $ \varnothing^2 {\color{blue}{\varnothing}} $ & $ 7 $ \\
$ 19 $ & $ 1680 $ & $  {\color{red}{\{123\,\,145\}}} $ & $  {\color{blue}{\{124678\,\,135678\}}} $ & $ \varnothing^4 $ & $ 4 $ \\
$ 20 $ & $ 10080 $ & $  {\color{red}{\{123\,\,145\}}} $ & $  {\color{blue}{\{124678\,\,234567\}}} $ & $ \varnothing^4 $ & $ 4 $ \\
$ 21 $ & $ 6720 $ & $  {\color{red}{\{123\,\,145\}}} $ & $  {\color{blue}{\{124678\,\,235678\}}} $ & $ \varnothing^3 {\color{blue}{\varnothing}} $ & $ 5 $ \\
$ 22 $ & $ 5040 $ & $  {\color{red}{\{123\,\,145\}}} $ & $  {\color{blue}{\{124678\,\,135678\,\,234567\}}} $ & $ \varnothing^5 $ & $ 2 $ \\
$ 23 $ & $ 840 $ & $  {\color{red}{\{123\,\,145\}}} $ & $  {\color{blue}{\{2345678\}}} $ & NB & $  {\color{red}{0}} $ \\
$ 24 $ & $ 840 $ & $  {\color{red}{\{123\,\,145\,\,167\}}} $ & $  {\color{blue}{\{\}}} $ & $ \varnothing^3 $ & $ 6 $ \\
$ 25 $ & $ 840 $ & $  {\color{red}{\{123\,\,145\,\,167\}}} $ & $  {\color{blue}{\{234567\}}} $ & $ \varnothing^4 $ & $ 4 $ \\
$ 26 $ & $ 5040 $ & $  {\color{red}{\{123\,\,145\,\,167\}}} $ & $  {\color{blue}{\{234568\}}} $ & $ \varnothing^3 {\color{blue}{\varnothing}} $ & $ 5 $ \\
$ 27 $ & $ 840 $ & $  {\color{red}{\{123\,\,145\,\,167\}}} $ & $  {\color{blue}{\{2345678\}}} $ & NB & $  {\color{red}{0}} $ \\
$ 28 $ & $ 6720 $ & $  {\color{red}{\{123\,\,145\,\,167\,\,246\}}} $ & $  {\color{blue}{\{\}}} $ & $ \varnothing^4 $ & $ 4 $ \\
$ 29 $ & $ 20160 $ & $  {\color{red}{\{123\,\,145\,\,167\,\,246\}}} $ & $  {\color{blue}{\{234578\}}} $ & $ \varnothing^4 {\color{blue}{\varnothing}} $ & $ 3 $ \\
$ 30 $ & $ 5040 $ & $  {\color{red}{\{123\,\,145\,\,167\,\,246\,\,257\}}} $ & $  {\color{blue}{\{\}}} $ & $ \varnothing^5 $ & $ 2 $ \\
$ 31 $ & $ 5040 $ & $  {\color{red}{\{123\,\,145\,\,167\,\,246\,\,257\}}} $ & $  {\color{blue}{\{345678\}}} $ & $ \varnothing^5 {\color{blue}{\varnothing}} $ & $ 1 $ \\
$ 32 $ & $ 1680 $ & $  {\color{red}{\{123\,\,145\,\,167\,\,246\,\,257\,\,347\}}} $ & $  {\color{blue}{\{\}}} $ & $ \varnothing^6 $ & $ 0 $ \\
$ 33 $ & $ 5040 $ & $  {\color{red}{\{123\,\,145\,\,167\,\,246\,\,257\,\,347\,\,358\}}} $ & $  {\color{blue}{\{\}}} $ & $ \varnothing^4\triangle^1 {\color{blue}{\triangle}} $ & $ 0 $ \\
$ 34 $ & $ 20160 $ & $  {\color{red}{\{123\,\,145\,\,167\,\,246\,\,257\,\,348\}}} $ & $  {\color{blue}{\{\}}} $ & $ \varnothing^3\triangle^1 {\color{blue}{\triangle}} $ & $ 2 $ \\
$ 35 $ & $ 10080 $ & $  {\color{red}{\{123\,\,145\,\,167\,\,246\,\,257\,\,348\,\,568\}}} $ & $  {\color{blue}{\{\}}} $ & $ \varnothing^2\triangle^2 {\color{blue}{\triangle}} $ & $ 1 $ \\
$ 36 $ & $ 5040 $ & $  {\color{red}{\{123\,\,145\,\,167\,\,246\,\,257\,\,478\}}} $ & $  {\color{blue}{\{\}}} $ & $ \varnothing^4\triangle^1 $ & $ 1 $ \\
$ 37 $ & $ 20160 $ & $  {\color{red}{\{123\,\,145\,\,167\,\,246\,\,258\}}} $ & $  {\color{blue}{\{\}}} $ & $ \varnothing^3\triangle^1 $ & $ 3 $ \\
$ 38 $ & $ 20160 $ & $  {\color{red}{\{123\,\,145\,\,167\,\,246\,\,258\}}} $ & $  {\color{blue}{\{345678\}}} $ & $ \varnothing^3\triangle^1 {\color{blue}{\varnothing}} $ & $ 2 $ \\
$ 39 $ & $ 20160 $ & $  {\color{red}{\{123\,\,145\,\,167\,\,246\,\,258\,\,357\}}} $ & $  {\color{blue}{\{\}}} $ & $ \varnothing^2\triangle^2 $ & $ 2 $ \\
$ 40 $ & $ 6720 $ & $  {\color{red}{\{123\,\,145\,\,167\,\,246\,\,258\,\,357\,\,368\}}} $ & $  {\color{blue}{\{\}}} $ & $ \varnothing^1\triangle^3 $ & $ 1 $ \\
$ 41 $ & $ 840 $ & $  {\color{red}{\{123\,\,145\,\,167\,\,246\,\,258\,\,357\,\,368\,\,478\}}} $ & $  {\color{blue}{\{\}}} $ & $ \triangle^4 $ & $ 0 $ \\
$ 42 $ & $ 20160 $ & $  {\color{red}{\{123\,\,145\,\,167\,\,246\,\,258\,\,378\}}} $ & $  {\color{blue}{\{\}}} $ & $ \varnothing^1\triangle^2 {\color{blue}{\triangle}} $ & $ 3 $ \\
$ 43 $ & $ 3360 $ & $  {\color{red}{\{123\,\,145\,\,167\,\,246\,\,357\}}} $ & $  {\color{blue}{\{\}}} $ & $ \varnothing^3\triangle^1 $ & $ 3 $ \\
$ 44 $ & $ 20160 $ & $  {\color{red}{\{123\,\,145\,\,167\,\,246\,\,358\}}} $ & $  {\color{blue}{\{\}}} $ & $ \varnothing^2\triangle^1 {\color{blue}{\triangle}} $ & $ 4 $ \\
$ 45 $ & $ 10080 $ & $  {\color{red}{\{123\,\,145\,\,167\,\,248\}}} $ & $  {\color{blue}{\{\}}} $ & $ \varnothing^2\triangle^1 $ & $ 5 $ \\
$ 46 $ & $ 10080 $ & $  {\color{red}{\{123\,\,145\,\,167\,\,248\}}} $ & $  {\color{blue}{\{234567\}}} $ & $ \varnothing^3\triangle^1 $ & $ 3 $ \\
$ 47 $ & $ 20160 $ & $  {\color{red}{\{123\,\,145\,\,167\,\,248\}}} $ & $  {\color{blue}{\{235678\}}} $ & $ \varnothing^2\triangle^1 {\color{blue}{\varnothing}} $ & $ 4 $ \\
$ 48 $ & $ 5040 $ & $  {\color{red}{\{123\,\,145\,\,167\,\,248\,\,358\}}} $ & $  {\color{blue}{\{\}}} $ & $ \varnothing^2\triangle^1 {\color{blue}{\triangle}} $ & $ 4 $ \\
$ 49 $ & $ 5040 $ & $  {\color{red}{\{123\,\,145\,\,167\,\,248\,\,358\}}} $ & $  {\color{blue}{\{234567\}}} $ & $ \varnothing^3\triangle^1 {\color{blue}{\triangle}} $ & $ 2 $ \\
$ 50 $ & $ 20160 $ & $  {\color{red}{\{123\,\,145\,\,167\,\,248\,\,368\}}} $ & $  {\color{blue}{\{\}}} $ & $ \varnothing^1\triangle^2 $ & $ 4 $ \\
$ 51 $ & $ 20160 $ & $  {\color{red}{\{123\,\,145\,\,167\,\,248\,\,368\}}} $ & $  {\color{blue}{\{234567\}}} $ & $ \varnothing^2\triangle^2 $ & $ 2 $ \\
$ 52 $ & $ 3360 $ & $  {\color{red}{\{123\,\,145\,\,167\,\,248\,\,368\,\,578\}}} $ & $  {\color{blue}{\{\}}} $ & $ \triangle^3 $ & $ 3 $ \\
$ 53 $ & $ 3360 $ & $  {\color{red}{\{123\,\,145\,\,167\,\,248\,\,368\,\,578\}}} $ & $  {\color{blue}{\{234567\}}} $ & $ \varnothing^1\triangle^3 $ & $ 1 $ \\
$ 54 $ & $ 3360 $ & $  {\color{red}{\{123\,\,145\,\,246\}}} $ & $  {\color{blue}{\{\}}} $ & $ \varnothing^3 $ & $ 6 $ \\
$ 55 $ & $ 10080 $ & $  {\color{red}{\{123\,\,145\,\,246\}}} $ & $  {\color{blue}{\{125678\}}} $ & $ \varnothing^4 $ & $ 4 $ \\
$ 56 $ & $ 10080 $ & $  {\color{red}{\{123\,\,145\,\,246\}}} $ & $  {\color{blue}{\{135678\}}} $ & $ \varnothing^3 {\color{blue}{\varnothing}} $ & $ 5 $ \\
$ 57 $ & $ 10080 $ & $  {\color{red}{\{123\,\,145\,\,246\}}} $ & $  {\color{blue}{\{125678\,\,134678\}}} $ & $ \varnothing^5 $ & $ 2 $ \\
$ 58 $ & $ 10080 $ & $  {\color{red}{\{123\,\,145\,\,246\}}} $ & $  {\color{blue}{\{125678\,\,345678\}}} $ & $ \varnothing^4 {\color{blue}{\varnothing}} $ & $ 3 $ \\
$ 59 $ & $ 3360 $ & $  {\color{red}{\{123\,\,145\,\,246\}}} $ & $  {\color{blue}{\{125678\,\,134678\,\,234578\}}} $ & $ \varnothing^6 $ & $ 0 $ \\
$ 60 $ & $ 840 $ & $  {\color{red}{\{123\,\,145\,\,246\,\,356\}}} $ & $  {\color{blue}{\{\}}} $ & $ \varnothing^4 $ & $ 4 $ \\
$ 61 $ & $ 2520 $ & $  {\color{red}{\{123\,\,145\,\,246\,\,356\}}} $ & $  {\color{blue}{\{125678\}}} $ & $ \varnothing^5 $ & $ 2 $ \\
$ 62 $ & $ 2520 $ & $  {\color{red}{\{123\,\,145\,\,246\,\,356\}}} $ & $  {\color{blue}{\{125678\,\,134678\}}} $ & $ \varnothing^6 $ & $ 0 $ \\
 \sout{$63$}  & \sout{$ 840 $} & \sout{$  {\color{red}{\{123\,\,145\,\,246\,\,356\}}} $} & \sout{$  {\color{blue}{\{125678\,\,134678\,\,234578\}}} $} & NR & NR  \\ \hline
\end{tabular}
\caption{Data for each candidate quatroid stratum. NR indicates \emph{not representable} and NB indicates \emph{not B\'ezoutian}. }
\label{tab:bigtable1}
\end{table}
}}
{\scriptsize{
\begin{table}[!htpb]
\begin{tabular}{|r|c|rl|c|c|}
 \hline 
 Num  & $ | $ Orb $ | $ &  Lines  &  Conics  &  Reducibles  & $ d_Q $\\ \hline \hline 
$ 64 $ & $ 10080 $ & $  {\color{red}{\{123\,\,145\,\,246\,\,357\}}} $ & $  {\color{blue}{\{\}}} $ & $ \varnothing^2\triangle^1 $ & $ 5 $ \\
$ 65 $ & $ 20160 $ & $  {\color{red}{\{123\,\,145\,\,246\,\,357\}}} $ & $  {\color{blue}{\{125678\}}} $ & $ \varnothing^3\triangle^1 $ & $ 3 $ \\
$ 66 $ & $ 10080 $ & $  {\color{red}{\{123\,\,145\,\,246\,\,357\}}} $ & $  {\color{blue}{\{125678\,\,134678\}}} $ & $ \varnothing^4\triangle^1 $ & $ 1 $ \\
$ 67 $ & $ 10080 $ & $  {\color{red}{\{123\,\,145\,\,246\,\,357\,\,678\}}} $ & $  {\color{blue}{\{\}}} $ & $ \triangle^2 {\color{blue}{\triangle}} $ & $ 5 $ \\
$ 68 $ & $ 10080 $ & $  {\color{red}{\{123\,\,145\,\,246\,\,378\}}} $ & $  {\color{blue}{\{\}}} $ & $ \varnothing^1\triangle^1 {\color{blue}{\triangle}} $ & $ 6 $ \\
$ 69 $ & $ 10080 $ & $  {\color{red}{\{123\,\,145\,\,246\,\,378\}}} $ & $  {\color{blue}{\{125678\}}} $ & $ \varnothing^2\triangle^1 {\color{blue}{\triangle}} $ & $ 4 $ \\
$ 70 $ & $ 5040 $ & $  {\color{red}{\{123\,\,145\,\,267\}}} $ & $  {\color{blue}{\{\}}} $ & $ \varnothing^1\triangle^1 $ & $ 7 $ \\
$ 71 $ & $ 20160 $ & $  {\color{red}{\{123\,\,145\,\,267\}}} $ & $  {\color{blue}{\{134678\}}} $ & $ \varnothing^2\triangle^1 $ & $ 5 $ \\
$ 72 $ & $ 5040 $ & $  {\color{red}{\{123\,\,145\,\,267\}}} $ & $  {\color{blue}{\{345678\}}} $ & $ \varnothing^1\triangle^1 {\color{blue}{\varnothing}} $ & $ 6 $ \\
$ 73 $ & $ 20160 $ & $  {\color{red}{\{123\,\,145\,\,267\}}} $ & $  {\color{blue}{\{134678\,\,234568\}}} $ & $ \varnothing^3\triangle^1 $ & $ 3 $ \\
$ 74 $ & $ 5040 $ & $  {\color{red}{\{123\,\,145\,\,267\,\,468\}}} $ & $  {\color{blue}{\{\}}} $ & $ \triangle^2 $ & $ 6 $ \\
$ 75 $ & $ 10080 $ & $  {\color{red}{\{123\,\,145\,\,267\,\,468\}}} $ & $  {\color{blue}{\{135678\}}} $ & $ \varnothing^1\triangle^2 $ & $ 4 $ \\
$ 76 $ & $ 5040 $ & $  {\color{red}{\{123\,\,145\,\,267\,\,468\}}} $ & $  {\color{blue}{\{135678\,\,234578\}}} $ & $ \varnothing^2\triangle^2 $ & $ 2 $ \\
$ 77 $ & $ 840 $ & $  {\color{red}{\{123\,\,145\,\,678\}}} $ & $  {\color{blue}{\{\}}} $ & $ \triangle^1 {\color{blue}{\triangle}} $ & $ 8 $ \\
$ 78 $ & $ 2520 $ & $  {\color{red}{\{123\,\,145\,\,678\}}} $ & $  {\color{blue}{\{234567\}}} $ & $ \varnothing^1\triangle^1 {\color{blue}{\triangle}} $ & $ 6 $ \\
$ 79 $ & $ 280 $ & $  {\color{red}{\{123\,\,456\}}} $ & $  {\color{blue}{\{\}}} $ & $ \triangle^1 $ & $ 9 $ \\
$ 80 $ & $ 2520 $ & $  {\color{red}{\{123\,\,456\}}} $ & $  {\color{blue}{\{124578\}}} $ & $ \varnothing^1\triangle^1 $ & $ 7 $ \\
$ 81 $ & $ 5040 $ & $  {\color{red}{\{123\,\,456\}}} $ & $  {\color{blue}{\{124578\,\,134678\}}} $ & $ \varnothing^2\triangle^1 $ & $ 5 $ \\
$ 82 $ & $ 1680 $ & $  {\color{red}{\{123\,\,456\}}} $ & $  {\color{blue}{\{124578\,\,134678\,\,235678\}}} $ & $ \varnothing^3\triangle^1 $ & $ 3 $ \\
$ 83 $ & $ 8 $ & $  {\color{red}{\{1234567\}}} $ & $  {\color{blue}{\{\}}} $ & NB & $  {\color{red}{0}} $ \\
$ 84 $ & $ 168 $ & $  {\color{red}{\{123456\,\,178\}}} $ & $  {\color{blue}{\{\}}} $ & NB & $  {\color{red}{0}} $ \\
$ 85 $ & $ 28 $ & $  {\color{red}{\{123456\}}} $ & $  {\color{blue}{\{\}}} $ & NB & $  {\color{red}{0}} $ \\
$ 86 $ & $ 280 $ & $  {\color{red}{\{12345\,\,1678\}}} $ & $  {\color{blue}{\{\}}} $ & NB & $  {\color{red}{0}} $ \\
$ 87 $ & $ 3360 $ & $  {\color{red}{\{12345\,\,167\,\,268\}}} $ & $  {\color{blue}{\{\}}} $ & NB & $  {\color{red}{0}} $ \\
$ 88 $ & $ 3360 $ & $  {\color{red}{\{12345\,\,167\,\,268\,\,378\}}} $ & $  {\color{blue}{\{\}}} $ & NB & $  {\color{red}{0}} $ \\
$ 89 $ & $ 840 $ & $  {\color{red}{\{12345\,\,167\}}} $ & $  {\color{blue}{\{\}}} $ & NB & $  {\color{red}{0}} $ \\
$ 90 $ & $ 56 $ & $  {\color{red}{\{12345\}}} $ & $  {\color{blue}{\{\}}} $ & NB & $  {\color{red}{0}} $ \\
$ 91 $ & $ 56 $ & $  {\color{red}{\{12345\,\,678\}}} $ & $  {\color{blue}{\{\}}} $ & NB & $  {\color{red}{0}} $ \\
$ 92 $ & $ 3360 $ & $  {\color{red}{\{1234\,\,1567\,\,258\,\,368\,\,478\}}} $ & $  {\color{blue}{\{\}}} $ & NB & $  {\color{red}{0}} $ \\
$ 93 $ & $ 10080 $ & $  {\color{red}{\{1234\,\,1567\,\,258\,\,368\}}} $ & $  {\color{blue}{\{\}}} $ & NB & $  {\color{red}{0}} $ \\
$ 94 $ & $ 5040 $ & $  {\color{red}{\{1234\,\,1567\,\,258\}}} $ & $  {\color{blue}{\{\}}} $ & NB & $  {\color{red}{0}} $ \\
$ 95 $ & $ 560 $ & $  {\color{red}{\{1234\,\,1567\}}} $ & $  {\color{blue}{\{\}}} $ & NB & $  {\color{red}{0}} $ \\
$ 96 $ & $ 840 $ & $  {\color{red}{\{1234\,\,156\,\,178\}}} $ & $  {\color{blue}{\{\}}} $ & NB & $  {\color{red}{0}} $ \\
$ 97 $ & $ 2520 $ & $  {\color{red}{\{1234\,\,156\,\,178\}}} $ & $  {\color{blue}{\{235678\}}} $ & NB & $  {\color{red}{0}} $ \\
$ 98 $ & $ 10080 $ & $  {\color{red}{\{1234\,\,156\,\,178\,\,257\}}} $ & $  {\color{blue}{\{\}}} $ & NB & $  {\color{red}{0}} $ \\
$ 99 $ & $ 10080 $ & $  {\color{red}{\{1234\,\,156\,\,178\,\,257\}}} $ & $  {\color{blue}{\{345678\}}} $ & NB & $  {\color{red}{0}} $ \\
$ 100 $ & $ 2520 $ & $  {\color{red}{\{1234\,\,156\,\,178\,\,257\,\,268\}}} $ & $  {\color{blue}{\{\}}} $ & NB & $  {\color{red}{0}} $ \\
$ 101 $ & $ 2520 $ & $  {\color{red}{\{1234\,\,156\,\,178\,\,257\,\,268\}}} $ & $  {\color{blue}{\{345678\}}} $ & NB & $  {\color{red}{0}} $ \\
$ 102 $ & $ 10080 $ & $  {\color{red}{\{1234\,\,156\,\,178\,\,257\,\,268\,\,358\}}} $ & $  {\color{blue}{\{\}}} $ & NB & $  {\color{red}{0}} $ \\
$ 103 $ & $ 5040 $ & $  {\color{red}{\{1234\,\,156\,\,178\,\,257\,\,268\,\,358\,\,467\}}} $ & $  {\color{blue}{\{\}}} $ & NB & $  {\color{red}{0}} $ \\
$ 104 $ & $ 20160 $ & $  {\color{red}{\{1234\,\,156\,\,178\,\,257\,\,358\}}} $ & $  {\color{blue}{\{\}}} $ & NB & $  {\color{red}{0}} $ \\
$ 105 $ & $ 20160 $ & $  {\color{red}{\{1234\,\,156\,\,178\,\,257\,\,358\,\,467\}}} $ & $  {\color{blue}{\{\}}} $ & NB & $  {\color{red}{0}} $ \\
$ 106 $ & $ 10080 $ & $  {\color{red}{\{1234\,\,156\,\,178\,\,257\,\,368\}}} $ & $  {\color{blue}{\{\}}} $ & NB & $  {\color{red}{0}} $ \\
$ 107 $ & $ 1680 $ & $  {\color{red}{\{1234\,\,156\}}} $ & $  {\color{blue}{\{\}}} $ & NB & $  {\color{red}{0}} $ \\
$ 108 $ & $ 5040 $ & $  {\color{red}{\{1234\,\,156\}}} $ & $  {\color{blue}{\{235678\}}} $ & NB & $  {\color{red}{0}} $ \\
$ 109 $ & $ 10080 $ & $  {\color{red}{\{1234\,\,156\,\,257\}}} $ & $  {\color{blue}{\{\}}} $ & NB & $  {\color{red}{0}} $ \\
$ 110 $ & $ 10080 $ & $  {\color{red}{\{1234\,\,156\,\,257\}}} $ & $  {\color{blue}{\{345678\}}} $ & NB & $  {\color{red}{0}} $ \\
$ 111 $ & $ 6720 $ & $  {\color{red}{\{1234\,\,156\,\,257\,\,358\}}} $ & $  {\color{blue}{\{\}}} $ & NB & $  {\color{red}{0}} $ \\
$ 112 $ & $ 20160 $ & $  {\color{red}{\{1234\,\,156\,\,257\,\,358\,\,467\}}} $ & $  {\color{blue}{\{\}}} $ & NB & $  {\color{red}{0}} $ \\
$ 113 $ & $ 6720 $ & $  {\color{red}{\{1234\,\,156\,\,257\,\,358\,\,678\}}} $ & $  {\color{blue}{\{\}}} $ & NB & $  {\color{red}{0}} $ \\
$ 114 $ & $ 6720 $ & $  {\color{red}{\{1234\,\,156\,\,257\,\,367\}}} $ & $  {\color{blue}{\{\}}} $ & NB & $  {\color{red}{0}} $ \\
$ 115 $ & $ 20160 $ & $  {\color{red}{\{1234\,\,156\,\,257\,\,368\}}} $ & $  {\color{blue}{\{\}}} $ & NB & $  {\color{red}{0}} $ \\
$ 116 $ & $ 5040 $ & $  {\color{red}{\{1234\,\,156\,\,257\,\,368\,\,478\}}} $ & $  {\color{blue}{\{\}}} $ & NB & $  {\color{red}{0}} $ \\
$ 117 $ & $ 10080 $ & $  {\color{red}{\{1234\,\,156\,\,257\,\,678\}}} $ & $  {\color{blue}{\{\}}} $ & NB & $  {\color{red}{0}} $ \\
$ 118 $ & $ 2520 $ & $  {\color{red}{\{1234\,\,156\,\,278\}}} $ & $  {\color{blue}{\{\}}} $ & NB & $  {\color{red}{0}} $ \\
$ 119 $ & $ 2520 $ & $  {\color{red}{\{1234\,\,156\,\,278\}}} $ & $  {\color{blue}{\{345678\}}} $ & NB & $  {\color{red}{0}} $ \\
$ 120 $ & $ 3360 $ & $  {\color{red}{\{1234\,\,156\,\,578\}}} $ & $  {\color{blue}{\{\}}} $ & NB & $  {\color{red}{0}} $ \\
$ 121 $ & $ 70 $ & $  {\color{red}{\{1234\}}} $ & $  {\color{blue}{\{\}}} $ & NB & $  {\color{red}{0}} $ \\
$ 122 $ & $ 420 $ & $  {\color{red}{\{1234\}}} $ & $  {\color{blue}{\{125678\}}} $ & NB & $  {\color{red}{0}} $ \\
$ 123 $ & $ 210 $ & $  {\color{red}{\{1234\}}} $ & $  {\color{blue}{\{125678\,\,345678\}}} $ & NB & $  {\color{red}{0}} $ \\
$ 124 $ & $ 280 $ & $  {\color{red}{\{1234\,\,567\}}} $ & $  {\color{blue}{\{\}}} $ & NB & $  {\color{red}{0}} $ \\
$ 125 $ & $ 35 $ & $  {\color{red}{\{1234\,\,5678\}}} $ & $  {\color{blue}{\{\}}} $ & NB & $  {\color{red}{0}} $ \\
$ 126 $ & $ 1 $ & $  {\color{red}{\{12345678\}}} $ & $  {\color{blue}{\{\}}} $ & NB & $  {\color{red}{0}} $ \\ \hline
 \end{tabular}
\caption{Data for each candidate quatroid stratum. NR indicates \emph{not representable} and NB indicates \emph{not B\'ezoutian}. }
\label{tab:bigtable2}
\end{table}
 }}

\newpage 
\begin{figure}[!htpb]
\includegraphics[scale=0.34]{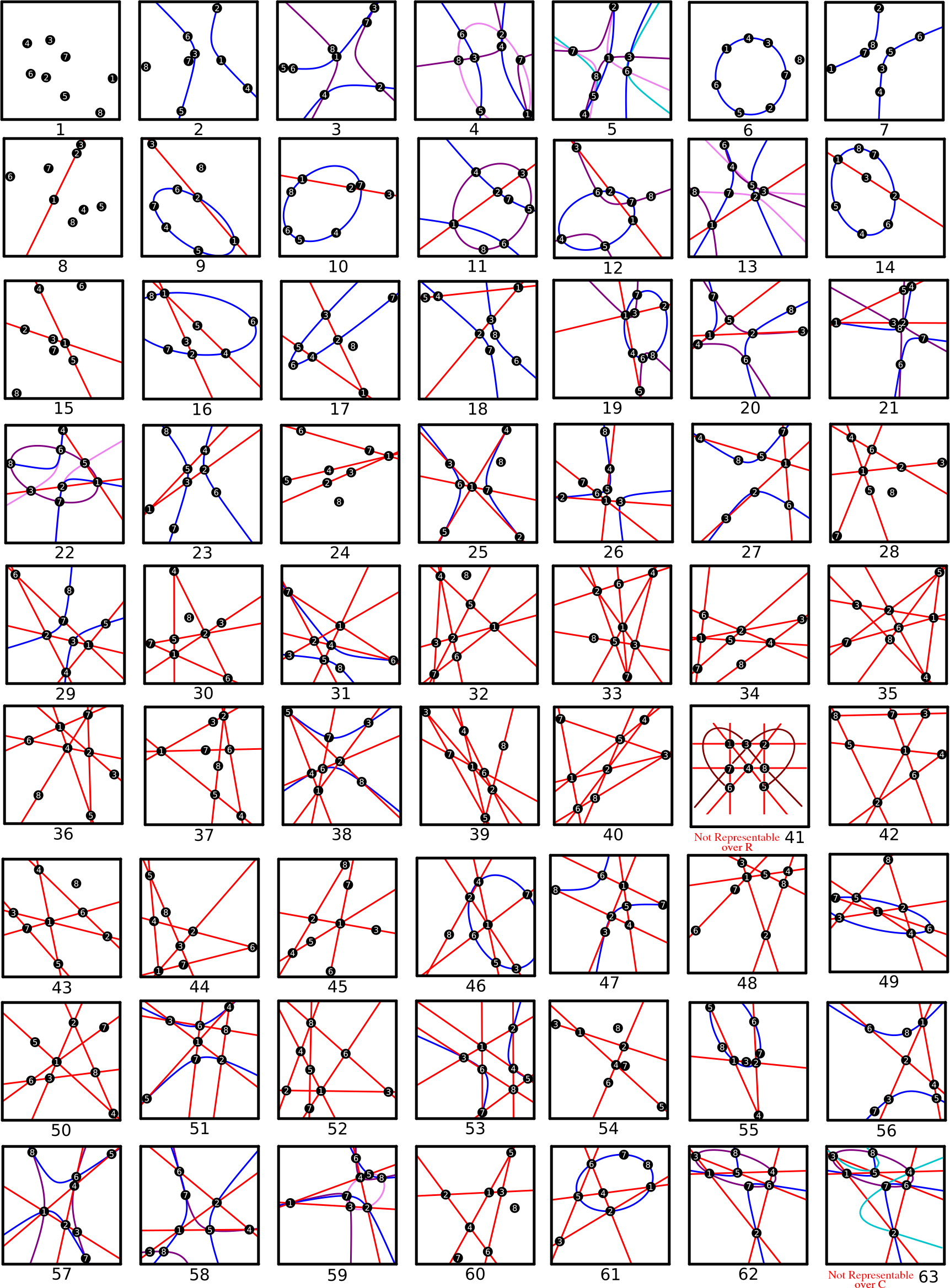}
\caption{Illustrations of candidate quatroids via real representatives if possible} 
\label{fig:Quatroids1}
\end{figure}
\begin{figure}[!htpb]
\includegraphics[scale=0.35]{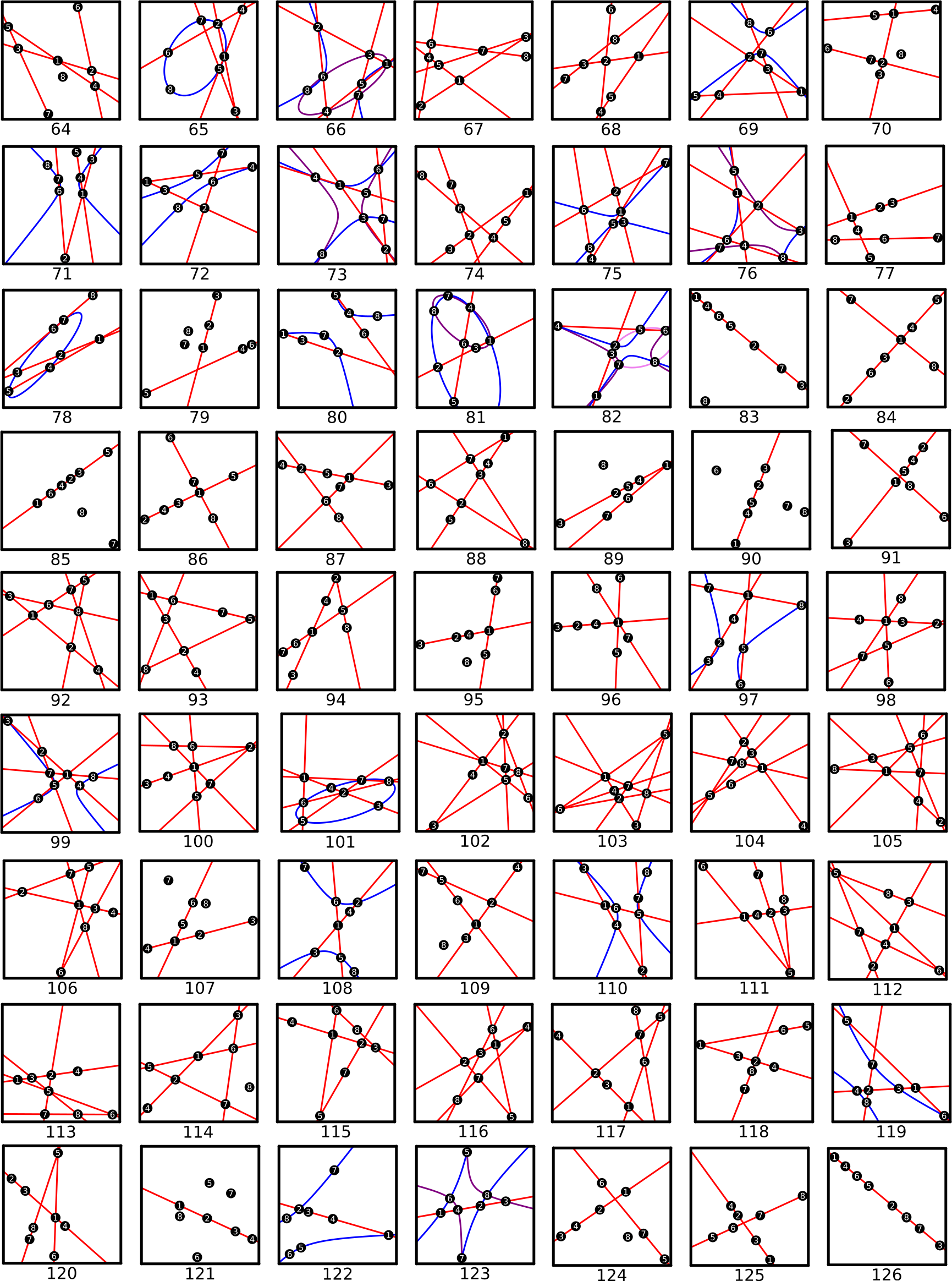}
\caption{Illustrations of candidate quatroids via real representatives if possible} 
\label{fig:Quatroids2}
\end{figure}

\newpage

{\small
\bibliographystyle{alphaurl}
\bibliography{curves.bib}
}
\end{document}